\documentclass[12pt,fleqn,reqno]{article}   
\usepackage[margin=1.35in]{geometry}
\usepackage{amsmath,amsthm,amssymb,mathrsfs,stmaryrd,mathtools}
\usepackage{enumerate}
\usepackage{color}
\usepackage{hyperref}

\usepackage{graphicx}

\usepackage{epsf}
\usepackage{epic}
\usepackage{makeidx}
\usepackage{psfrag}

\theoremstyle{plain}
\newtheorem{theorem}{Theorem}
\newtheorem{proposition}[theorem]{Proposition}
\newtheorem{lemma}[theorem]{Lemma}
\newtheorem{corollary}[theorem]{Corollary}

\theoremstyle{definition}

\newtheorem{remark}[theorem]{Remark}

\numberwithin{theorem}{section} 
\numberwithin{equation}{section}

\newcommand{\ba}[1]{\begin{array}{#1}}
\newcommand{\ea}{\end{array}}

\newcommand{\be}{\begin{equation}}
\newcommand{\ee}{\end{equation}}
\newcommand{\bea}{\begin{eqnarray}}
\newcommand{\eea}{\end{eqnarray}}
\newcommand{\beann}{\begin{eqnarray*}}
\newcommand{\eeann}{\end{eqnarray*}}

\newcommand{\R}{\mathbb{R}}
\newcommand{\C}{\mathbb{C}}
\newcommand{\Z}{\mathbb{Z}}

\newcommand{\bH}{{\mathbb{H}}}

\newcommand{\cA}{\mathcal{A}}

\newcommand{\cE}{\mathcal{E}}

\newcommand{\nc}{\newcommand}

\nc{\al}{\alpha}
\nc{\bet}{\beta}
\nc{\G}{\Gamma}
\nc{\g}{\gamma}
\nc{\del}{\delta}
\nc{\lam}{\lambda}
\nc{\Om}{\Omega}
\nc{\Omt}{\tilde{\Omega}}
\nc{\ta}{\tau}
\nc{\w}{\omega}
\nc{\om}{\omega}
\nc{\io}{\iota}
\nc{\h}{\theta}
\nc{\z}{\zeta}
\nc{\s}{\sigma}
\nc{\Si}{\Sigma}
\nc{\Lam}{\Lambda}

\nc{\vphi}{\varphi}

\nc{\ut}{t} 
\newcommand{\an}{a^n}

\newcommand{\harm}{\Om^1_{\rm harm}}

\newcommand{\onef}{\Om^1}

\newcommand{\re}{\operatorname{Re}}
\newcommand{\im}{\operatorname{Im}}
\renewcommand{\Re}{\operatorname{Re}}
\renewcommand{\Im}{\operatorname{Im}}
\newcommand{\Null}{\operatorname{Null}}

\nc{\lra}{\leftrightarrow}
\nc{\ra}{\rightarrow}
\newcommand{\Ran}{{\rm Ran\,}}
\nc{\ran}{\rangle}
\nc{\lan}{\langle}

\newcommand{\one}{{\bf{1}}}
\nc{\bfone}{{\bf 1}}

\nc{\p}{\partial}
\nc{\pt}{\partial_t}
\nc{\ptt}{\partial_t^2}
\nc{\dt}{\partial_t}

\nc{\dA}{\nabla_A}
\newcommand{\n}{\nabla}

\newcommand{\DETAILS}[1]{}

\title{Ginzburg--Landau equations on Riemann surfaces of higher genus}
\author{D. Chouchkov\footnote{Dept. of Math., U. of Toronto, Toronto, ON, M5S 2E4, Canada}, N. M. Ercolani\footnote{Dept. of Math., U. of Arizona, Tucson, AZ, 85721-0089, USA}, S. Rayan\footnote{Dept. of Math. \& Stats., U. of Saskatchewan, Saskatoon, SK, S7N 5E6, Canada}, I. M. Sigal\footnote{Dept. of Math., U. of Toronto, Toronto, ON, M5S 2E4, Canada}}
\date{April 23, 2019} 

\begin{document}
\newcommand\numberthis{\addtocounter{equation}{1}\tag{\theequation}}

\maketitle

\abstract{We study the Ginzburg--Landau equations 
 on Riemann surfaces of arbitrary genus.  In particular,  we

\begin{itemize} 
\item construct explicitly  the (local moduli space of gauge-equivalent) solutions in the neighbourhood of the constant curvature ones; 
 
\item classify holomorphic structures on line bundles arising as solutions to the equations in terms of the degree, the Abel-Jacobi map, and symmetric products of the surface;   

\item determine the form of the energy and identify when it is below the energy of the constant curvature (normal) solutions.

\end{itemize} 
\medskip
 
Nous \'etudions les \'equations de Ginzburg-Landau d\'efinies sur des surfaces de Riemann de genre arbitraire. En particulier, 

\begin{itemize} 
\item  nous construisons explicitement l'espace des modules locaux des solutions (\'equivalentes par transformation de jauge) dans le voisinage des solutions de courbure constante; 

\item nous classifions les structures holomorphiques dans les fibr\'es en droites qui apparaissent comme solutions de ces \'equations, en fonction de leur degr\'e, de l'application d'Abel-Jacobi, et des produits sym\'etriques de surface; 

\item nous obtenons une expression pour l'\'energie et identifions dans quelles conditions elle est inf\'erieure � l'\'energie des solutions (normales) de courbure constante. 
\end{itemize}

\section{The Ginzburg--Landau Equations} \label{sec:GLE}

Let $X$ be a Riemann surface  with a complex structure given by a hermitian metric $h$ and let $E$ be  a smooth, unitary line bundle over $X$.  
The Ginzburg-Landau equations on $X$ involve a section, $\psi$, and a connection one-form $a$ on $E$  and  are written as
 \begin{align}\label{GLE} \begin{cases} 
 \Delta_a\psi &=  
 \kappa^2(|\psi|^2 - 1)\psi, \\
d^*d a &= 
 \Im(\overline{\psi}\nabla_a \psi),
\end{cases}\end{align}
where $\nabla_a$  and $-\Delta_a=\nabla_a^* \nabla_a$ are the covariant derivative and Laplacian associated with a unitary connection form $a$ (locally, $\nabla_a = d + ia$ ) and  $\kappa^2$ is the Ginzburg-Landau parameter. The adjoints here are taken with respect to the inner products on sections and bundle-valued one-forms, respectively, induced by fixing a smooth hermitian inner product, $k$ (which determines the $U(1)$-structure of $E$) on the fibers of $E$ and a hermitian metric, $h$ on $X$ (see e.g. \cite{Donaldson, Jost}).



Equations \eqref{GLE} lie at the foundation of the macroscopic theory of superconductivity. Soon after their birth, they migrated to other areas of condensed matter physics and then to particle physics. They gave the first example of gauge field theory, led to the Yang-Mills-Higgs equations and, together with the latter, formed a foundation 
 of the standard model of particle physics.  (For some background, see \cite{JT, MSu, Wit}.)

Equations \eqref{GLE} are the Euler-Lagrange equation of the Ginzburg-Landau energy functional \begin{align}\label{GLen}
	\cE (\psi, a, h) &= \|\nabla_a \psi\|^2  + \|d a\|^2 + \frac{\kappa^2}{2}\|(|\psi|^2-1)\|^2. 
\end{align}

\medskip

On $\R^2$ (the cylindrical geometry, in the physics literature), equations \eqref{GLE} are translation, rotation and gauge invariant (see \eqref{gauge-transf} below). However, for low 
 magnetic fields, its ground state - the solution with the lowest energy per unit area (see below) - turns out not to be gauge - translation invariant. This was discovered by A.A. Abrikosov who suggested that for Type II superconductors it has the symmetry of a lattice. 
 
 The Abrikosov's prediction was confirmed in experiments and was recognized by a Nobel Prize.  The corresponding state is called the Abrikosov, or vortex, lattice. 
As alluded above, it is the ground state for Type II superconductors. 
   This phenomenon is related to the one of the crystallization and proving the vortex lattice solutions are ground states is a major open problem. 

\medskip

The Abrikosov solutions can be reformulated as solutions of \eqref{GLE} on a flat torus. Hence, mathematically, it is natural to go one step further and consider \eqref{GLE} on an arbitrary smooth Riemann surface.  
(On the physics side, one can imagine superconducting thin membranes with surfaces of higher genera. In this case,  \eqref{GLE} would have to be modified, but the present mathematical theory would still apply.)

We fix a genus $g$ Riemann surface with a fixed homology basis (referred to as a marking).
Recall that  a hermitian metric, $h$, determines the Riemann metric, volume form (as the real and imaginary parts of $h$) and the complex structure (relating the first two). We allow $h$ to vary keeping the complex structure fixed within its conformal class. By a solution of \eqref{GLE} we understand the {\it triple} $(\psi, a, h)$, 
(or the equivalence classes - or moduli - of such triples related by the gauge transformation, see below).  


To fix ideas, in what follows, we allow variations of  Hermitian metrics only along a family, $h=r h', r>1,$ where $h'$ is an arbitrary fixed metric. 

To formulate our result, we need some notation and definitions:

The curvature of a  connection $a$ is defined as $F_a= d a$ and a  connection $a$ on $E$ is said to be of {\it constant curvature} iff 
$F_a$ is of the form $F_a  = 
 b\omega$,  where $b$ is a constant and $\omega$ is the symplectic {\it volume form} on $X$.  By the Chern-Weil relation (see \eqref{flux-quant} below), $b = \frac{2\pi n }{|X|}$, where $|X|$ denotes the total area of $X$;

 deg$(E)$ denotes the degree (a topological invariant with values in $\Z$) of $E$; 
 we assume $E$ has the degree $n$ and denote this topological bundle, unique up to homeomorphism, by $E_n$. 

   $\mathcal{A}_{cc, n}$ denotes the space of constant curvature unitary (with respect to $k$) connections on $E_n$; 

  $\mathscr{H}^r_n$ and $\vec{\mathscr{H}^r}$  denote the order $r$ Sobolev spaces of 
sections of $E_n$ and co-closed (i.e. $d^*\al=0$) one-forms, $\al\in \onef$ (for $r<1$, the derivative $d^*\al$ is defined in the weak sense);

$\mathcal{O}_1$ and $\mathcal{O}_2$ stand for error terms in the sense of the norms in $\mathscr{H}^2$ and $ \vec{\mathscr{H}^2}$, 
 respectively.

 The main results of this paper can be summarized as follows

\begin{theorem}\label{thm:GLE-RS-exist} 
Let $1 \leq n \leq g$ and
 $a^c \in \mathcal{A}_{cc, n}$ 
  be a regular value of the Abel-Jacobi map $\Phi$  (see \eqref{AJ-map} below). 
 Assume $r$ is close to $2\pi n/|X|$ and 
\begin{equation}  \label{nec-cond}
(r-b)(\kappa^2 -\kappa^2_c(a^c)) > 0. 
\end{equation}
 with $\kappa^2_c(a^c)$ defined below. Then  
the space, $H^1_{\rm hol}$, of holomorphic sections  of the holomorphic bundle corresponding to $a^c$ (see Theorems \ref{thm:bundle-holomorphiz} and \ref{thm:KM} below) has the dimension $1$; the Ginzburg--Landau equations \eqref{GLE} on $E_n$ have  the branch of solutions 
\begin{align}\label{sol-general'} (\psi_{s}, a_{s} , r_s) &= \big(s \phi + \mathcal{O}_1(s^3), a^c + \mathcal{O}_2(s^2), 2\pi n/|X| + \mathcal{O}(s^2)\big), \end{align}
where 
$  \phi=g {\bf\hat{\phi} }$, with $g$ a gauge transformation (constructed explicitly) and $\hat{\phi} \in H^1_{\rm hol}$,  
normalized as $\|\phi\|_{\mathscr{H}^2}^2/|X|=1$, and  $s\in \R$  is given by the equation, with $\beta(a^c)$ defined below,  
\begin{equation}    \label{s-r}
         s^2 
       = \frac{ \kappa^2(r-b)}{ \beta(a^c)(\kappa^2 -\kappa^2_c(a^c))} +O(\kappa^4(r-b)^2);     \end{equation}
 up to gauge equivalence, these are the only solutions in $\mathscr{H}^2 \times \vec{\mathscr{H}^2}\times \R$ in a small neighbourhood of the solution branch $\{(0, a^c,  r):  r >0\}$.  	
\end{theorem}

Note that \eqref{nec-cond} follows from \eqref{s-r}. 
Taking $s=0$ in \eqref{sol-general'} gives a normal solution $(0, a^c, h)$. In fact, it is well known, see e.g. \cite{Wells}, that 
  \begin{align}  \label{norm-sol} &(0, a, h) \text{ is a solution of 
   \eqref{GLE} on    $E_n$    iff }  a \in \cA_{cc, n} 
      \end{align}
 The solution $(0, a, h)$, where  $a$ is a constant curvature connection on $E$, will be called the {\it normal state} (irreducible solution in some mathematics literature).

 Now, we define the Abrikosov function $ \beta(a^c)$ and the parameter threshold $\kappa_c(a^c)$, $ a^c\in \mathcal{A}_{cc, n}$, used in the theorem, by: 
\begin{equation} \label{beta'}
    \beta(a^c) :=  \min_{\phi\in \Null \p_{a^c}'',\ \lan |\phi|^2 \ran=1} \lan |\phi|^4 \ran, 
\end{equation}
 where $\p_{a}''$ stands for the $(0, 1)-$part of  the unitary connection $\n_{a}$, and $\langle f \rangle:=\frac{1}{|X|}\int_{X}f$, 
 and 
  \begin{equation}  \label{kappa-c}
\kappa_c(a^c):= \sqrt{\frac{1}{2} \bigg(1 -\frac{1}{\beta(a^c)}\bigg)} 
\end{equation} 
Note $\kappa_c(a^c) < 1/\sqrt{2}$ and recall that $\kappa = 1/\sqrt{2}$ is the self-duality point (see e.g. \cite{JT}). 

$\beta(a^c)$ is invariant under the gauge transformations (see Proposition \ref{prop:gauge-sym} below) and therefore is in fact defined on the space of holomorphic structures on $E_n$ (see Theorem \ref{thm:KM} of Appendix \ref{sec:admissible-conn}).

Furthermore, if  $a^c \in \mathcal{A}_{cc, n}$ is a regular value of the Abel-Jacobi map $\Phi$, then $\Null \p_{a^c}''$ is one-dimensional (see Proposition \ref{prop:classif-adm-con}) and 
\begin{equation} \label{beta}\beta(a^c) :=   \lan |\phi|^4 \ran,\end{equation} with $\phi\ \in \Null \p_{a^c}''$ and normalized as $ \lan |\phi|^2 \ran=1$.

For a review and references in the torus case see \cite{S}.  For general compact Riemann surfaces, the existence of solutions of  \eqref{GLE} for $\kappa=1/\sqrt 2$ (the (anti)self-dual case) was shown in 
\cite{Bradlow, Bradlow2, BradlDask, Bradlow3, GarPrada} (see also \cite{MRi}, the results of \cite{Bradlow, Bradlow2, BradlDask, Bradlow3} apply actually to the Yang-Mills-Higgs equations). 

Furthermore, \cite{Nagy} (see also \cite{Qing}) gives the existence proof by the variational technique, which by its nature is not constructive. 

Theorem \ref{thm:GLE-RS-exist} reveals that the structure of the solution depends crucially on whether the corresponding line bundles have holomorphic sections and, if they do, on the dimension of the space they span.

Our next result addresses the question of the energy of the bifurcating solution. 
\begin{theorem}\label{thm:ener-asymp} Let $b = \frac{2\pi n }{|X|}$. We have, under the conditions of Theorem \ref{thm:GLE-RS-exist}, 
 \begin{equation}     \label{Eexpan}
    \frac{1}{|X|}    \mathcal{E} (\psi_{s}, a_{s}) =         \frac{\kappa^2}{2} +(2\pi n)^2    - \frac12 \frac{\kappa^4(r-b)^2}{\beta(a^c) (\kappa^2 -  \kappa_c^{2}(a^c))}
          +  O(\kappa^6(r-b)^6),
           \end{equation}
 Note that $ \frac{\kappa^2}{2} +(2\pi n)^2$ is the energy of the  normal state $(0, a^1)$ per unit area.   \end{theorem}

\begin{corollary}\label{cor:ener-beta} 
  We have for $s$ sufficiently small
 \begin{align}     \label{Ebeta'}
& \mathcal{E}(\psi_{s}, a_{s})  < \mathcal{E}(0, a^c)\ \Longleftrightarrow \ 
  \kappa^2 > \kappa_c^{2}(a^c),\\
  &   \inf\mathcal{E}_{E}(\psi, a)  <  \inf\mathcal{E}_{E'}(\psi', a') \ \Longleftrightarrow \ 
 \beta(a^c) <\beta((a^c)').      \end{align}
\end{corollary}

We say that a solution $u_*=(\psi_{*}, a_{*})$ is {\it stable} iff Hess$ \mathcal{E}(u_*)\ge 0$.  
It follows from our results on the linearized problem (see Proposition \ref{prop:Delta_a-part_a} and Eq \eqref{dF-norm}), see also \cite{Nagy}, that the solution $(0, a, h)$, where  $a$ is a constant curvature connection on $E$ with $F_a=  \frac{2\pi n }{|X|}\omega$, is stable iff $ |X|\le  2\pi n/\kappa^2$.

Following the Hessian computations in \cite{ST1}, one can show that $(\psi_{s}, a_{s}, r_s)$ is stable iff $r_s\ge 2\pi n/|X|$.

We discuss Theorem \ref{thm:GLE-RS-exist} and its proof. 
  We begin with a key fact that \eqref{GLE} is a gauge theory. 
In particular, it has the gauge symmetry, which can be formulated as following proposition which can be checked directly.

\begin{proposition}\label{prop:gauge-sym}
If $g: E \rightarrow E'$ is a $U(1)$-equivariant isomorphism of smooth line bundles over a Riemann surface $X$ and $(\psi, a)$ solve \eqref{GLE} on $E$, then \begin{align}\label{gauge-transf} T_g^{\rm gauge} (\psi, a)= (g\psi,  a + i g^{-1}d g)\end{align} 
 is a section and connection pair on $E'$ and it solves \eqref{GLE} on $E'$.\footnote{If general, if $g: E \rightarrow E'$ is an isomorphism of smooth line bundles over a Riemann surface $X$, then for any section and connection pair, $(\psi, a)$ on $E$, $(g\psi , a+ i g^{-1}dg)$
 is a section and connection pair on $E'$.}  
\end{proposition}

For $E'=E$ maps \eqref{gauge-transf} 
are called the gauge transformations. We use the latter term for general $g: E\ra E'$. 
For $E'=E$, the above means that a single solution is simply a representative of an infinite dimensional equivalence class of solutions. Therefore when we describe a solution with some property, we mean that this property holds up to a gauge equivalence  (with $E'=E$).

 \medskip

Topologically equivalent classes of smooth line bundles over a Riemann surface are determined by their degree with values in $\Z$. However, a bundle, $E\equiv E_n$, of a given degree $n$ may be given a variety of distinct  holomorphic structures. 

 A {\it unique equivalence class of holomorphic structures} on $E_n$ is determined by $\p_{a^c}''$ 
   (see Appendix \ref{sec:admissible-conn},  Theorem \ref{thm:KM}) and $a^c \in \mathcal{A}_{cc, n}$.

 We derive Theorem \ref{thm:GLE-RS-exist} from some general facts from the theory of Riemann surfaces presented below and the following key result:

\begin{theorem}\label{thm:GLE-RS-exist-cond} 
Let $a^c \in \mathcal{A}_{cc, n}$ and assume the space $\Null\p_{a^c}''$ is one-dimensional. 
Then the statement of Theorem \ref{thm:GLE-RS-exist} is true.
\DETAILS{Let $h=r h'$, where $r\ge 1$ and $h'$ is a fixed hermitian metric. Then 
the Ginzburg--Landau equations \eqref{GLE} on $E$ have solutions of the form 
\begin{align}\label{sol-general'} (\psi_{s}, a_{s} , r_s) &= \big(s \phi + \mathcal{O}_1(s^3), a^c + \mathcal{O}_2(s^2), 1 + \mathcal{O}(s^2)\big). \end{align}
where  $s\in \R$ is sufficiently small, $ \phi$ is a solution to $\p_{a^c}''\phi=0$ and $\mathcal{O}_1$ and $\mathcal{O}_2$ are in the sense of the norms in $\mathscr{H}^2$ and $ \vec{\mathscr{H}^2}$, 
 respectively. 
 Up to the gauge equivalence, these are the only solutions in $\mathscr{H}^2 \times \vec{\mathscr{H}^2}\times \R$ in a small neighbourhood of the solution branch $\{(0, a,  r):  r \ge 1\}$.}  	
\end{theorem}
 Theorem \ref{thm:GLE-RS-exist-cond}, and therefore  Theorem \ref{thm:GLE-RS-exist}, could be readily generalized to the case when  the space $\Null\p_{a^c}''$ is odd-dimensional.

We now consider the dimension of the space $\Null\p_{a^c}''$.  
 On the first step, we use the following result, which is a constructive version 
  of known results (see Appendix \ref{sec:admissible-conn}, Theorem \ref{thm:KM} (i)): 
\begin{theorem} \label{thm:bundle-holomorphiz}
There is a smooth, invertible map from  $U(1)$ bundles with unitary connections  onto holomorphic bundles with  complex connections and compatible hermitian metrics, which maps the former connections into the latter ones.  This map is constructed explicitly.  
\end{theorem}
 We give a hands on proof of this theorem  in Appendix \ref{sec:bundle-holomorphiz}.

By  Theorem \ref{thm:bundle-holomorphiz}, 
the space $\Null\p_{a^c}''$ is isomorphic to the space of  holomorphic sections of the holomorphic bundle corresponding to $a^c$. 
We say a holomorphic structure, or its corresponding connection, $a^c$, is {\it admissible} if the kernel of $\p_{a^c}''$ is exactly one-dimensional.\footnote{In the physics literature such connections are called {\it non-degenerate}.} We will give a complete classification of the space of admissible connections including
\begin{enumerate}
\item determinaion of its degree range $n$  (Proposition \ref{prop:classif-adm-con}); 
\item its description as an analytic moduli space  (Corollary \ref{cor:Null-one-d});
\item explicit formulas for the elements of  $\Null\p_{a^c}''$ in terms of normalized differentials of the third kind  (Appendix \ref{sec:Null-expl}). 
\end{enumerate}

This detailed classification is based entirely on classical results, though their application here to the Ginzburg-Landau equations is novel. This will be explained below.  
We now state our principal classification. 
\begin{proposition}\label{prop:classif-adm-con} 
The following statements give a complete classification of admissible connections and the possible degrees of their associated line bundles. 

(i) For either $n < 0$ or $n > g$, there are no admissible holomorphic structures $a_c$. 

(ii) For $1 \leq n \leq g$, a holomorphic structure $a^c \in \mathcal{A}_{cc, n}$ 
  is admissible if and only if it corresponds to a regular value of the Abel-Jacobi map $\Phi$ (see \eqref{AJ-map} below) restricted to $X^{(n)}$. 

(iii) For $n = 0$, only the trivial bundle supports admissible connections and $\phi$ in this case may be gauge transformed to a unitary constant, $c$ with $|c| = 1$. 
\end{proposition}
In the first statement, $n < 0$ is ruled out because if  a line bundle 
 had a holomorphic section, its divisor of zeroes would be of non-negative degree which is the same as the degree, $n$, of the bundle. 
Hence, one must have $n \geq 0$. 


 For $n > g$, Riemann's inequality states that $h^0(X,E)$ $:=\dim H^0(X,E)$ satisfies $h^0(X,E) \geq n - g + 1 > 1$ which is not admissible. Hence.  one must have $n \leq g$. 

The second statement follows from Abel's theorem  as discussed in Appendix \ref{sec:admissible-conn}. For the last statement, meromorphic sections of a degree zero bundle have the same number of zeroes as poles.  Hence a global holomorphic section has no zeroes and so the bundle must be trivializable.


Let $X^{(n)}$ denote the $n$-fold symmetric product of $X$ with itself (this is a smooth, compact complex manifold, in fact a projective algebraic variety, of complex dimension $n$).  The main ingredient in the proof of Proposition \ref{prop:classif-adm-con} is the {\it Abel-Jacobi map}, $\Phi$, 
extended naturally to 
 $X^{(n)}$ and defined there as 
\begin{align} \label{AJ-map} 
&\Phi: X^{(n)} \to  Jac(X) := \mathbb{C}^g/\Lambda, \notag\\
& P_1 + \cdot + P_n  \to  \sum_{k=1}^n \int_{P_0}^{P_k} \vec{\zeta}, 
\end{align}
where $\vec{\zeta} = (\zeta_1, \dots, \zeta_g)$ is a basis for the space of holomorphic differentials on $X$, normalized with respect to a canonical homology basis, $a_1, \dots, a_g; b_1, \dots, b_g$, meaning $\oint_{a_j} \zeta_k = \delta_{jk}$, and     $\Lam$ is the rank $2g$ lattice generated by the  periods, $\delta_{jk}$ and $\tau_{ij} = \int_{b_j} \zeta_i$, of $\vec{\zeta}$.
The map depends on the choice of basepoint $P_0$, but for various $P_0$'s, differs only by a translation in $Jac(X)$ under a change of this basepoint.


The set of admissible connections $a^c$, of degree $n$, has the structure of a complex manifold which can be described in terms of the Abel-Jacobi image, $W_n := \Phi(X^{(n)})\subset Jac(X)$,  
 which is an analytic (in fact algebraic) subvariety of $Jac(X)$ (see Appendix \ref{sec:admissible-conn}). Denote by $W_{n, \text{smooth}}$ the set of the regular values of $\Phi$ in $W_n$. The singular points of $W_n$ correspond precisely to the holomorphic structures on $E$ for which $h^0(X,E) > 1$. Again this follows from Abel's theorem \cite{FarKra}, III.11.12. So we have

\begin{corollary} \label{cor:Null-one-d} 
The space of holomorphic line bundles with the underlying smooth bundle $E$ and having one-dimensional spaces of holomorphic sections 
 is parametrized by the complex sub-manifold $W_{n, \text{smooth}} \subset Jac(X)$.  Since it is the set of regular values of $\Phi$, by Sard's theorem, this latter space is an open dense submanifold of $W_n$.

\end{corollary}
Theorems \ref{thm:bundle-holomorphiz} and \ref{thm:GLE-RS-exist-cond} and Corollary \ref{cor:Null-one-d} imply  Theorem \ref{thm:GLE-RS-exist}. 
  (Theorem \ref{thm:bundle-holomorphiz} connects the original $U(1)-$bundle considered in Theorem \ref{thm:GLE-RS-exist-cond} to a holomorphic one from the space described in Corollary \ref{cor:Null-one-d}.)

\begin{remark}\label{rem:gauge-equiv-sect} There is a one-to-one correspondence between homomorphisms of $E$ and sections of line bundles: a homomorphism $E\ra E'$ is a section of the line bundle $E'\otimes E^{-1}$ and a section, $g$, of a line bundle $J$  defines the homomorphism  $g: E\ra E':=J\otimes E$. 
\end{remark} 
\DETAILS{A uniformization of $X$ on its universal covering space  has the advantage of lifting our equations (\ref{GLE}) to a simply connected domain.  where they can be thought about independently of the complex structure on $X$; i.e., independently of $\G$.}

Theorem \ref{thm:GLE-RS-exist} can be formulated and proved entirely on the trivial bundle, $\tilde E:= \tilde X \times \mathbb{C}$, over the universal cover, $\tilde X$, of $X$. While the approach taken in this paper is more economical, its  uniformized version is more explicit. We present some details in Appendix \ref{sec:equiv-formul}.

 
The second reason in pursuing the  uniformized approach is that it offers a much more general definition of the (dynamical) stability solutions of the (lifting of the) GLE \eqref{GLE}, which disrupts the relation between the description on $X$ and its universal cover $\tilde X$ (see Proposition \ref{prop:proj-lift-sol}). This will be explained in more detail elsewhere.  
 

\bigskip

Considering a family, $h=r h', r>1,$ of  Hermitian metrics, 
it is convenient to rescale the Ginzburg-Landau equations \eqref{GLE} to the fixed hermitian metric, $h'$. Then  \eqref{GLE} rescale to the form 

  \begin{align}\label{GLE'} \begin{cases} -\Delta_a\psi &=  (\mu- \kappa^2|\psi|^2) \psi,\\ 
d^*d a &= 
 \Im(\overline{\psi}\nabla_a \psi),
\end{cases}\end{align}
where  $ \mu:=\kappa^2 r$ (see Appendix \ref{sec:rescGLE}). From now on, we consider \eqref{GLE'}, with $\mu$ an arbitrary positive number,  and  omitting the prime, the hermitian metric $h$. 

\medskip

The paper is organized as follows. After a preliminary Sections \ref{sec:flux-quant-gen} and Section \ref{sec:normal-branch}, we prove 
Theorem \ref{thm:GLE-RS-exist-cond}, in Sections \ref{sec:lin-probl} and \ref{sec:bifurc}, with some technical derivations given in Appendix \ref{sec:Weitz-form}. In Section \ref{sec:en-expan} we prove Theorem \ref{thm:ener-asymp}. 

\medskip

 In  Appendix \ref{sec:equiv-formul} we present an approach to the problem on the universal cover of $X$ and in Appendices \ref{sec:bundle-holomorphiz} and  \ref{sec:admissible-conn} and \ref{sec:Null-expl}, we  prove Theorem \ref{thm:bundle-holomorphiz} and discuss the proof of Proposition \ref{prop:classif-adm-con} and an explicit form of Corollary \ref{cor:Null-one-d}, 
  respectively. 
 
 We also have Appendix \ref{sec:bundleEF}, 
 which gives some additional results not used in the proofs and comments.

\section{Flux quantization} 
\label{sec:flux-quant-gen}
We have the following well-known result (the magnetic flux quantization or the Chern-Weil relation):

\begin{theorem}\label{thm:flux-quant-gen} 
If $E$ is a line bundle on $X$, with a connection $a$, then
\begin{eqnarray}\label{flux-quant}\frac{1}{2\pi} \int_X F_a &=  \deg(E) \in  \Z,\end{eqnarray}
where $\deg(E)$ is the degree of $E$. 
\end{theorem}  
 $c_1(E):=\frac{1}{2\pi} \int_X F_a$ is called sometimes the first Chern number. It can be defined in terms of the automorphy factor of $X$, see \cite{Gun2}. 

\bigskip

Now, suppose $\nabla_a$ is a connection of constant curvature on $E$. 
Then, by \eqref{flux-quant}, we have
 \begin{align}  \label{b-area-gen}
 b = \frac{1}{|X|}\int_X F_a   = 2\pi c_1(E)/|X|= 2\pi n/|X|. 
\end{align}
 where $|X|$ is the area of $X$. For the hyperbolic metric the Gauss-Bonnet formula gives that 
$|X|=2\pi(2g-2)$. This, together with 
\eqref{b-area-gen}, gives
 \begin{align}  \label{b-area-hyperb}
 b =  n/(2g-2). 
\end{align}

\section{The normal branch of solutions} \label{sec:normal-branch}

For reference reason we present the next well-known fact (see e. g. \cite{Huy}, Proposition 4.2.2 and \cite{DonKron}, Sect 2.2.1, page 50) as
  \begin{lemma}\label{lem:normal-sol}  
 Any constant curvature connection, $a$, on $E$ is of the form  $a_{*} +\beta, $ where $a_{*}$ is a fixed constant curvature connection on $E$ and $\beta$ is a closed one-form, $ d\beta =0$.
\end{lemma}
Next, the important statement \eqref{norm-sol} is implied by the following

\begin{lemma}\label{lem:MaxwAb}  $a$ is an constant curvature connection on $E$ 
iff it satisfies the (Maxwell) equation 
\begin{align}\label{Maxw-ab} d^*d a &= 0.\end{align}
Consequently, the triple $(0, a, \mu)$ solves the Ginzburg-Landau equations  \eqref{GLE'} for every constant curvature connection $a$ and every $\mu>0$, which proves \eqref{norm-sol}. 
\end{lemma}
 \begin{proof} 
 Let $*$ be  the Hodge star operator (the properties $*$ are listed in \eqref{star}). Using the relation  $d^*=(-1)^k *d*$ on $k-$forms, we find $d^*d a=- *d*d a$. If $*d a= b$ for some constant $b$, 
then $d^*d a=- *d b =0$. 
  
  Conversely, if $d^*d a = 0$, then $d f= 0$, where $f:=*d a$ is a function on $X$, which implies that $f$ is constant. \end{proof}
     
Recall that we call the branch  $\{(0, a,  \mu): a$ is a constant curvature connection 
and $\mu>0\}$ of solutions of the Ginzburg-Landau equations \eqref{GLE'}, the {\it normal} (or constant curvature) branch. We will bifurcate from this branch. 

\bigskip

Let  
 $\Om^1_{c}$  
be the subspace of real (smooth or square integrable) closed one-forms  on $X$ with values in $ E^*_n$.  
Lemma \ref{lem:normal-sol} implies that the space, $\cA_{cc, n}$, of constant curvature connections on $E_{n}$ is of the form 
 \begin{align}\label{ccc-space} \cA_{cc, n} = a^{n}+\Om^1_c,\end{align}
  where $a^{n}$ is  a fixed constant curvature connection on $E_{n}$. By \eqref{flux-quant}, $d a^n=2\pi n \om$. 

\section{The linear problem}\label{sec:lin-probl}

The first step in the bifurcation analysis is to investigate the linearized equations.

We linearize the Ginzburg-Landau equations \eqref{GLE'} at the solution $(0, a^n)$,  where $a^n$ is a constant curvature connection on $E_n$
  (see  Lemmas \ref{lem:MaxwAb} and \ref{lem:normal-sol}),  to obtain the pair of equations 
\begin{align}\label{GLElin}
-\Delta_{a^n}\phi &-  \mu \phi=0, \
d^*d \alpha = 0,
\end{align}
on $X$.  
  Our goal is to solve these equations. We begin with the first equation.

\bigskip

We  decompose, as usual, the unitary connection $\n_a$ as $\n_a=\partial_{a}'+  \partial_{a}''$, where $\partial_{a}': \Om^{p, q} \ra  \Om^{p+1, q}$ and $\partial_{a}'': \Om^{p, q} \ra  \Om^{p, q+1}$ (cf. Section  \ref{sec:cc-con-exist}). We show 

 \begin{proposition}\label{prop:Delta_a-part_a} Let $a^n$ be a connection of the constant curvature, 
 $F_a=b \om\ (b= \frac{2 \pi n }{|X|})$. Then the operator $-\Delta_{\an}$ has purely discrete spectrum, is bounded below as $-\Delta_{a^n} \ge b$ and $b$ is  the smallest eigenvalue of $-\Delta_{\an}$ iff  $\Null \partial_{\an}''\ne \emptyset$. More precisely, 
  \begin{align}\label{Delta_a-part_a}\Null (-\Delta_{a^n} -  b) = \Null \partial_{a^n}''.
   \end{align} 
\end{proposition}
The value $\mu=b$ is a  bifurcation point for our problem. More precisely, if 
 $\Null \partial_{a^n}''\ne \emptyset$, then $\Null (-\Delta_{a^n} -  b) =\emptyset$, if  $\mu<b$, 
  and $\Null (-\Delta_{a^n} -  b) \ne \emptyset$ if $\mu=b$.

The statement that the operator $-\Delta_{a^n}$ has purely discrete spectrum is a standard consequence of the condition that $X$ is compact. The rest of Proposition \ref{prop:Delta_a-part_a} follows from 
 the following Weitzenb\"ock-type formula or harmonic oscillator representation, which was proven in \cite{Tejero} and whose simple proof is presented in Appendix \ref{sec:Weitz-form}.

 \begin{proposition}\label{prop:Delta_a-repr}
 \begin{align}
{\partial_a''}^* \partial_a'' &= \frac{1}{2}(-\Delta_a - *F_a).\label{Lapl-formula}
\end{align} \end{proposition} 
Since $*F_{a^n}=b$, this proposition implies Proposition \ref{prop:Delta_a-part_a}.

\bigskip

This finishes the first equation in \eqref{GLElin}. For the second equation in \eqref{GLElin}, let  $\Om^1_{\rm harm}$ be the subspace of real (smooth or square integrable) one-forms  on $X$ with values in $ E^*_n$, which are harmonic, i.e.  satisfying 
  \begin{align}\label{harm-df} d \beta = 0,\  d^*\beta = 0.    \end{align}
   Then, by Lemma \ref{lem:MaxwAb}, the definition of $\vec{\mathscr{H}^2}$ and the integration by parts ($\lan \beta, d^*d\beta\ran =\|d\beta\|^2$), we have
  \begin{corollary}\label{cor:null d* d=harm} 
 \begin{align}\label{null2}\Null  d^*d\big|_{\vec{\mathscr{H}^2}} =\harm. \end{align} 
   \end{corollary}
\paragraph{Remark.}  Let  $H^1_{dR}(\mathbb{H},\mathbb{R})$ be the first de Rham cohomology group. Then we have the following standard result (which follows from the Hodge decomposition theorem) 
 \begin{align}\label{harm-dR}\Om^1_{\rm harm}(X, \mathbb{R}) \approx  H^1_{dR}(X, \mathbb{R}), 
   \end{align} 
  where $\approx$ stands for the isomorphism identifying elements of $\Om^1_{\rm harm}(X, \mathbb{R})$ with equivalence classes of $H^1_{dR}(X, \mathbb{R})$.

\section{Bifurcation analysis }\label{sec:bifurc}


We define the spaces, $\mathscr{L}^2_n$ and $\vec{\mathscr{L}}^2$, of $L^2-$sections of $E_{n}$ and (weakly) co-closed $L^2-$one-forms 
 on $X$ (with values in $End(E_n)$) 
 and let $\mathscr{H}^s_{n}$ and $\vec{\mathscr{H}^s_{n}}$  be their natural Sobolev versions. 
 
\medskip

 We look for solutions of  the Ginzburg--Landau equations, \eqref{GLE'}, in the form  $(\psi, a^{n} +\alpha)$, where $(\psi, \alpha) \in \mathscr{H}^2_{n} \times \vec{\mathscr{H}^2_{n}}$. Using  \eqref{Maxw-ab} ($d^*d a^{n} = 0$), we rewrite  \eqref{GLE'} as
  \begin{align} \label{GLE''} \begin{cases} &-\Delta_{a^n + \alpha}\psi + (\kappa^2|\psi|^2 - \mu)\psi=0, \\ 
&d^*d \alpha -\Im(\overline{\psi}\nabla_{a^{n} + \alpha}\psi)=0.
\end{cases}\end{align}
Our strategy is, following \cite{TS, 
 CSS}, to solve first the following modified equation 
\begin{align}  \label{F-eq} & F(\psi, \alpha, \mu) =  0,\ F :\mathscr{H}^2 \times \vec{\mathscr{H}^2} \times \mathbb{R}^+ \rightarrow \mathscr{L}^2_{n} \times \vec{\mathscr{L}^2_{n, \s}}, \\
  \label{F} & F(\psi, \alpha, \mu) =  (-\Delta_{a^{n} + \alpha}\psi + (\kappa^2|\psi|^2 - \mu)\psi,\  d^*d \alpha -  P_{\rm co-clo} J (\psi, \al)),
\end{align}
where $J (\psi, \al):=\Im(\overline{\psi}\nabla_{a^{n } + \alpha}\psi)$ and $P_{\rm co-clo}$ is the orthogonal projection in the space of real square integrable one-forms onto the subspace of real co-closed one-forms, $\Om^1_{c*}$. 
 Then we go from \eqref{F-eq} back to the original system \eqref{GLE''} using the following proposition and its corollary:
 \begin{proposition}\label{prop:J-co-closed}  
Assume $(\psi, a)$ solves the first equation in \eqref{GLE''}.
\DETAILS{, i.e. 
\begin{align}  \label{GLE1}  -\Delta_{a}\psi + (\kappa^2|\psi|^2 - \mu)\psi=0.
\end{align}}
 Then $ J (\psi, a)$ is a co-closed one-form (i.e. $d* J=0$).
   \end{proposition}
 \begin{corollary}\label{cor:solGLE} A solution $(\psi, \alpha, \mu)$ to \eqref{F-eq} solves also \eqref{GLE''}.    \end{corollary}
 \begin{proof}[Proof of Proposition \ref{prop:J-co-closed}]  Given a base point $z_0\in \tilde X$ and a closed form $\beta$ on $X$, 
we define the family of maps
\begin{align}  \label{gs-tilde}\tilde g_s(z) \equiv \tilde g(z)^s,\  \tilde g(z)\equiv \tilde g_{z_0,   \beta}(z):= e^{ i\  \int_{z_0}^z\tilde \beta}: \tilde X\ra U(1),\end{align}
where $\tilde \beta$ is a lift of $\beta$ to $\tilde X$  and $\int_{z_0}^z\tilde \beta$ is the integral of $\tilde \beta$ over  a  path, $c_{z_0, z}$, in $\tilde X$ from $z_0$ to $z$. Since $\beta$ is a closed one-form, these maps are independent of the choice of paths $c_{z_0, z}$ from the same homotopy class, are differentiable and satisfy  
\begin{align}  \label{gs-tilde-cocycle}\tilde g_s(\g z) = \tilde g_s(z) \s'_s(\g),\ \s'_s(\g):=e^{ i\ s \int_{\g } \beta}\in Hom(\G, U(1)).\end{align} 
Here $\int_{\g } \beta$ is a period of $\beta$. 
 (For the last equation, we have $\tilde g_s(\g z)  \tilde g_s(z)^{-1}= e^{ i\ s \int_{z}^{\g z}\tilde \beta}= e^{ i\ s \int_{\g } \beta}$. 
 For more details, see \cite{Forster}, Theorem 10.13.)

As a consequence of the last equality,  the maps $\tilde g_s(z)$ descend to sections $g_s$ of a line bundle over $X$ with fibers isomorphic to $U(1)$, which satisfy $g_s^{-1} d g_s= i s \beta$.

The sections $g_s$ 
 generate the gauge transformations 
  and  the rescaled Ginzburg-Landau energy functional 
\begin{align}\label{GLen}
	\cE (\psi, a, h) &= \|\nabla_a \psi\|^2  + \|d a\|^2 + \frac{\kappa^2}{2}\|(|\psi|^2-\mu/\kappa^2)\|^2. 
\end{align}    
  is invariant under these transformations: $\cE (\psi, a) =\cE (g_s\psi, a+ i g_s^{-1} d g_s)$. We differentiate the last equation 
 with respect to $s$ at $s=0$, to obtain 
\begin{align} \label{E-deriv'}\Re\lan \n_{a}\psi, \n_{a}(\ln g \psi)\ran & + \lan(\kappa^2|\psi|^2 - \mu)\psi,  \ln g \psi\ran\notag \\ 
& + \lan d^* d a -J (\psi, a),\beta\ran=0,\end{align}
Assume $\beta$ is exact (and therefore its periods vanish). Then we can write $g=e^{i f}$, where $f$ is a real (single-valued) function on $X$ and integrating by parts in the first term and using $d a=d\al$ and $ \lan d^* d\al, \beta\ran=\lan d\al, d\beta\ran=0$, we find 
\begin{align} \label{E-deriv} \Re\lan -\Delta_{a}\psi  &+ (\kappa^2|\psi|^2 - \mu)\psi, i f \psi\ran- \lan J (\psi, a),\beta\ran=0.\end{align} 
This, together with the first equation in \eqref{GLE''}, 
  implies $\lan J (\psi, a),\beta\ran=0$. 
 Since the last equation holds for any exact form $\beta$, we conclude that $ J (\psi, a)$ is co-closed. 
\end{proof}

Thus, by Corollary \ref{cor:solGLE}, it suffices to solve the equation $F=0$, with $F$ given in \eqref{F}. Let $u:=(\psi, \al)$.
The map, $F(u, \mu)$, has the properties (a) $F$ is $C^2$; (b) $F(T^{\rm gauge}_{\eta}u, \mu) =F(u, \mu)$ and (c)  $F(0, \mu) = 0,  \forall \mu$, and therefore $(0,\mu)$ is a trivial branch of solutions to $F(u, \mu) =0$.\footnote{In fact, there is a larger trivial branch of the form $(\psi=0, \beta, \mu)$, where $\beta \in \harm$, but we will avoid using the latter.}

Clearly, \begin{align}\label{dF-norm}d F (0, \mu)= (-\Delta_{a^{n}} -  \mu) \oplus d^*d.\end{align}
 By Proposition \ref{prop:Delta_a-part_a}, $\Null (-\Delta_{a^n} -  b) = \Null \partial_{a^n}''$ and by \eqref{null2}, we have $\Null_{\vec{\mathscr{H}^2}} d^*d = \harm$.  
Hence  the null space,  $K_{\mu}$, of $d F (0, b)$ is 
\begin{align}\label{NulldF'}  K_{ \mu} = 
\begin{cases} & \Null \p_{\an }'' \times \harm\ 
 \text{ if }\ \mu = b,\\ 
 & \{0\} \times \harm\ \text{ if }\ \mu < b.
\end{cases}\end{align}
We define $P$ to be the projection on $K_{ b}$ and let $P^\perp = 1- P $. By \eqref{NulldF'} and the definition of $P^\perp$,  $d_w P^\perp F\big|_{w=0}= P^\perp d_wF\big|_{w=0}$ is invertible on $\Ran  P^\perp$. Now, we perform Lyapunov-Schmidt reduction:
\begin{align}
P F(v + w, \mu) = 0,  \label{ker}\\
P^\perp F(v + w, \mu) = 0, \label{perp}
\end{align}
where $v:=P (\psi, \alpha)$ and $w:=P^\perp (\psi,\alpha)$. Since $d_w P^\perp F\big|_{w=0}$ is invertible on $\Ran  P^\perp$, provided $v$ and $\mu - b$ are sufficiently small,  we can use the implicit function theorem to find a unique solution, $w=w(v, \mu)$, for \eqref{perp} (again provided $v$ and $\mu - b$ are sufficiently small). 

\bigskip

By the assumption of Theorem \ref{thm:GLE-RS-exist-cond}, $\p_{\an }''$  has a non trivial one-dimensional kernel. Let $\phi$ be the unique (modulo gauge transformations) solution to the equation $\p_{\an}''\phi=0$. 
 We assume $\phi $ is normalized, i.e. $\|\phi \|=1$. 
Observe that 
\begin{align} \label{v}v\equiv v (s,  \mu) := P (\psi, \alpha)=(s\phi, \beta), \end{align}  
   where $s=\lan \phi, \psi\ran$ and $\beta$ is the projection of $\al$ onto $\harm$.  
  So that the solution to \eqref{perp} can be written as $w=w(s, \beta, \mu)$, with $(s, \beta, \mu)$ in a neighbourhood of $(0, 0, b)$.

To estimate $w(s, \mu)$, we decompose $P^\perp F$ into the linear and nonlinear components, $P^\perp F(v + w, \mu)=L^\perp w + N(v, w)$, and rewrite \eqref{perp} as the fixed point problem $ w = - (L^\perp)^{-1} N(v, w)$. Then using estimates on $N(v, w)$ and its derivatives, it is not hard to show that,  for $(s, \mu)$ in a neighbourhood of $(0, b)$,  
\begin{align} \label{w-deriv}\p^\nu w(s, \beta, \mu) 
=(O(s^3), O(s^2)),\ \forall \nu,\end{align} 
where $\p^\nu$ is a derivative of the order $\nu$ in $s, \beta,   \mu$ (recall that $\beta\in\harm$, a finite dimensional vector space) and  $O(s^k)$ stands for a remainder, which is uniformly bounded, together with its derivatives in $\mu$, by $C s^k$. 

Substituting the solution $w=w(s, \beta, \mu)$ 
  to \eqref{perp} into equation \eqref{ker}, we arrive at the equation 
\begin{align}\label{PF-eq} P F(s, \beta, \mu) =0,\ F(s, \beta, \mu) :=  F(v(s, \beta, \mu) + w(s, \beta, \mu), \mu). 
\end{align}

This equation has the $U(1)$ symmetry in $s$: for $\al$ constant, we have $P F(e^{i \al} s, \beta, \mu) =e^{i \al} P F(s, \beta, \mu) $. \DETAILS{We fix the $U(1)$ gauge by 
 restricting $\Null \p_{\an}''$ to a one dimensional real vector space by using the property (b) to pass from $s\phi $ to $|s| \phi $. {\bf (do we need this?)}} 

Let $F=(F_1, F_2)$ and $P_{\rm harm}$ be the the second component of  the projection $P$, i.e. the orthogonal projection onto $\harm$.  By the definition of $P$, the equation $P F(s, \beta,  \mu)=0$ is equivalent to the equations $\lan\phi, F_1(s, \beta,  \mu)\ran =0$ and $P_{\rm harm} F_2(s, \beta,  \mu) =0$. By \eqref{v} and \eqref{w-deriv},  $\psi=s \phi +O(s^3)$ and $\al=O(s^2)$. Using this,  we  expand  $F_i(s, \beta,  \mu)$  in $s$ to the leading order:
\begin{align}\label{F1-expan} 
& F_1(s, \beta,  \mu) =  s(-\Delta_{\an + \beta}   - \mu \mathbb{I}) \phi +O(s^3),\\ 
\label{F2-expan}& F_2(s, \beta,  \mu) = |s|^2[- \im (\bar \phi \n_{a^n} \phi)+ |\phi|^2 \beta]+O(s^4).  \end{align}
 In the next proposition, we show that the leading term on the r.h.s. of \eqref{F2-expan} drops out.  Recall the decomposition  $\n_a=\p_{a}' + \p_{a}''$. 
   \begin{proposition}\label{prop:Jn-co-exact} Let $a^n$ be a constant curvature connection on $E_n$ and $\phi$ is a solution to $\p_{a^n}''\phi=0$.  Then  $ J (\phi, a^n)$ is co-exact and consequently 
  \begin{align} \label{ProjJ} P_{\rm harm}\im (\bar \phi \n_{a^n} \phi)=0. \end{align}     \end{proposition}
\begin{proof}  
Using that $\n_a=\p_{a}' + \p_{a}''$ and that $\phi$ is a solution to $\p_{a^n}''\phi=0$, we find
\begin{align}
\Im(\overline{\phi}\nabla_{a^n}\phi) &= \Im(\overline{\phi}[\partial_{a^n}' + \partial_{a^n}'']\phi) \\
 &= \Im(\overline{\phi}\partial_{a^n}'\phi).
\end{align}
Next,  we can proceed either locally or by lifting to the universal cover and using 
 $\n_A=\p_{A_c}' + \p_{A_c}''$, where   $A_c:=\frac12 (A_1-iA_2) dz,\ \partial_{A_c}' = \partial + i A_c$   and $\p_{A_c}'' = \overline{\partial} + i \bar A_c = -\p_{A_c}^*$ (see Appendix \ref{sec:general-cover} and equations \eqref{pA} and \eqref{compl-conn}). We take the former route.

Using that $\overline{\phi}\partial\phi =\partial|\phi|^2 - \phi \partial \overline{\phi}$, we find $\overline{\phi}\partial_{a^n}'\phi= \partial|\phi|^2 - \phi (\partial - i a^n)\overline{\phi} =\partial|\phi|^2$. Next,  $\Im \p f =\frac12 * d f$ (to see this we recall that locally or on the universal cover, we have $\p := \frac{\partial}{\p z}\otimes d z$, where, as usual, $ \frac{\partial}{\p z}:=  \p_{z}:=(\partial_{x_1} - i \partial_{x_2})/2$, and $*^2=-\one$, $*(dx_1) = dx_2$ and $*(dx_2) = -dx_1$). The last two relations give
\begin{align}\label{J-co-exact} \Im(\overline{\phi}\partial_{a^n}'\phi) 
&= \Im(\partial|\phi|^2) = \frac12 *d|\phi|^2 = -\frac12 d^*(*|\phi|^2)
\end{align}
Hence  $ J (\phi, a^n)$ is co-exact. \end{proof} 
Hence using  \eqref{F1-expan} - \eqref{F2-expan}, \eqref{ProjJ}, $\Delta_{a^n + \beta}  \phi=\Delta_{a^n}  \phi + O(s\beta) $ and $(-\Delta_{a^n} - b) \phi=0$, we find
\begin{align} \label{PF-expan}
P F(s, \beta, \mu) = \big(O(s\beta) + & s (\mu - b) \phi 
+O(s^3),\\  \label{PF-expan'}&|s|^2P_{\rm harm}(|\phi|^2 \beta)+O(s^4)\big). \end{align}

Recall that $\harm$ is a $2g-$dimensional space. We fix an orthonormal basis $\{\omega_i\}_{i=1}^{2g}$ for $\harm$ so that $\beta=\sum_i t_i\omega_i$  and let  $\ut:=(t_1, \dots, t_{2g})$. Then 
  \[P (\psi, \alpha)=(s\phi, \sum_i t_i \omega_i).\] 
    Then the solution for $(\ref{perp})$ can be written in the form of $w\equiv w(s, \ut, \mu)$, 
 for $s, \ut, \mu$ in a neighbourhood of $(0, 0, b)$. 

 Let  $v(s, \ut, \mu) := P(\psi, \alpha)=(s\phi, \sum_i t_i\omega_i)$. Substituting the solution $w=w(s, \ut, \mu) = (\psi^\perp(s, \ut, \mu), \alpha^\perp(s, \ut, \mu))$ to \eqref{perp} into equation \eqref{ker}, we arrive at the equation $G(s, \ut, \mu) =0$, where
\begin{align}
G(s, \ut, \mu) := PF(v(s, \ut, \mu) + w(s, \ut, \mu), \mu). 
\end{align}
Let  $G=(G_1, G_2)$. The equations $G_1(s, \ut, \mu)=0$ and $G_2(s, \ut, \mu)=0$ are equivalent to the equations $\lan\phi, F_1(s, \ut, \mu)\ran =0$ and $\lan\om_i, F_2(s, \ut, \mu)\ran =0, i=1, \dots, 2g,$ which, after using \eqref{PF-expan} - \eqref{PF-expan'} and division by $s$ and $|s|^2$, respectively, give 
\[(b-\mu)\|\phi\|^2+O(|t|) +O(s^2)=0,\]\[ \sum_j B_{ij} t_i 
  +O(s^2)=0, \forall i,\]
where $ B_{ij} := \langle \omega_j, |\phi|^2 \omega_i \rangle_X$ and $O(s^2)$ stand for a remainder uniformly bounded, together with its derivatives in $\ut, \mu$, by $C s^2$.  
Clearly, the matrix $ B_{ij}$ is strictly positive and therefore is invertible. 
Hence, by the implicit function theorem, the latter system of equations has a unique solution for $(\mu, t)$ of the form 
\begin{align} \label{mu-t-solns}\mu=b+O(s^2),\  
t= 
 O(s^2). 
 \end{align}
Thus, equation \eqref{PF-eq} (or $G(s, \ut, \mu) =0$) has a solution for $\mu$ and $\ut$ as a function of $s$. Together with the previous conclusion, this gives the solution to  the equation $F(u, \mu) =0$ for $u:=(\psi, \alpha)$, $\mu$ and $\ut$ in terms of $s$. 
 Thus we have proven
  
\begin{proposition} \label{prop:bifurc-sol}
 For $s$ sufficiently small, 
the equation $F(u, \mu) =0$ has a family of solutions of the form $u_s=(s \phi +O(s^3), O(s^2)),\  \mu_s=b\kappa^2+O(s^2)$. 
By Corollary \ref{cor:solGLE}, 
   $u_s$ solve \eqref{GLE''}. 
     \end{proposition}
Therefore \eqref{GLE'} has a solution of the form \eqref{sol-general'} with the properties stated in Theorem \ref{thm:GLE-RS-exist}, i.e. that the new solutions bifurcate from the trivial one. 
Next, we prove \eqref{s-r}. It will follow from 

\begin{proposition} \label{prop:leadingorder}
Let $b:=2\pi n/|X|$. The bifurcating solutions,  $(\psi_s,\ a_s,\ \mu_s)$, constructed in Proposition \ref{prop:bifurc-sol}, 
have the following expansions 
\begin{equation} \label{epsexpansions}
\psi_{s} = s\phi+ O(s^3),\ a_s =a^n +  s^2 a_{1} +  O(s^4), \ \mu_{s} = b \kappa^2 + s^2\mu_1 + O(s^4),  
\end{equation}
where the first and the second remainder are understood in the sense of the norms in 
$\mathscr{H}^2_{n}$ and $\vec{\mathscr{H}^2_{n}}$, respectively (with the first and second derivatives of the remainders obeying similar estimates). Moreover, $\phi$, $a_{1}$ and $\mu_1$ satisfy the equations 
\begin{align} \label{phi-a1-eqs} 
-\Delta_{a^{n}} \phi =  b \phi\ \mbox{and}\ 
d a_{1}= \frac{1}{2} * (1 
 -  |\phi |^2),\ 
\end{align}
and 
   \begin{equation} \label{lam1}
\mu_1=\left[\frac{1}{2} + \left(\kappa^2-\frac{1}{2}\right)\beta(a^n)\right], 
\end{equation}
where 
$\beta(a^n)$ is the Abrikosov function given by \eqref{beta}. 
   \end{proposition}
\begin{proof} The Lyapunov-Schmidt arguments of the main proof (more precisely, the equations $w = - (L^\perp)^{-1} N(v, w),\ N=(N_1, N_2)$ and $N_2=-J(\psi, \al)$) give \eqref{epsexpansions}. Plugging \eqref{epsexpansions} into \eqref{GLE'} 
and taking $s \rightarrow 0$ gives the first equation in \eqref{phi-a1-eqs} 
 and
\begin{equation}
\label{a1-eq'}
d^* d a_{1} =  \Im (\bar{\phi} \nabla_{a^{n}}
\phi ).
\end{equation}
This together with \eqref{J-co-exact} gives $d^* (d a_{1}+ \frac{1}{2}  *|\phi|^2)=0$. Hence 
\begin{align} \label{a1-eq''} 
d a_{1}=h -\frac{1}{2} * |\phi |^2,\ 
\end{align}
where $h$ is a two form obeying $d^* h=0$ and therefore $h$ is constant. Since $\int_X d a_{1} =0$, we have $h=*\frac12 $, which gives the second equation in \eqref{phi-a1-eqs}. 

Now we prove \eqref{lam1}. First we note that by the self-adjointness of  the operator $-\Delta_{a} $,   expansions \eqref{epsexpansions} and $(-\Delta_{a^n}  - b)\phi=0$ (see \eqref{phi-a1-eqs}. 
\[\lan \phi,  (-\Delta_{a} - \mu)\psi \ran =s^2[ \lan \phi,  2i a_{1}\cdot\n_{a^{n}}\psi \ran -\mu_1 \lan \phi,  \psi \ran] +O(s^4).\]  
Next, we multiply the first equation in \eqref{GLE'}
scalarly (in $L^2(X)$) by $\phi$, use the above relation, 
substitute the expansion for $\psi$ in \eqref{epsexpansions} and take $s =0$, to obtain
  \begin{align} \label{solvcond'}
		 \lan \phi,  2i a_{1}\cdot\n_{a^{n}}\phi \ran -\mu_1 \lan \phi,  \phi \ran+ \kappa^2\lan \phi,  |\phi|^2 \phi\ran=0.
	\end{align}
	This expression implies that  the imaginary part of  the first term on the left hand side of \eqref{solvcond'} is zero. (We arrive at the same conclusion by integrating by parts - using the (gauge-) periodicity of  $\psi_{0}$ and $ a_{1}$ -  and using that $a_{1}$, like $\al$, is co-closed, i.e. $d^*a_{1}=0$.)    Therefore
    \begin{align*}
  \lan \phi,  2i a_{1}\cdot\n_{a^{n}}\phi \ran =       2i \int_{X} \bar{\phi} a_{1} \wedge *\n_{a^{n}}\phi 
                    &= -2 \int_{X} a_{1} \wedge * \Im( \bar{\phi} \n_{a^{n}}\phi ) .  
         \end{align*}
Using the second equation in \eqref{GLE'} in the last term and integrating by parts, we obtain $ \lan \phi,  2i a_{1}\cdot\n_{a^{n}}\phi \ran = -2  \lan  a_{1}, d^* da_{1} \ran = -2 \lan d a_{1}, da_{1} \ran.$ Next, using \eqref{phi-a1-eqs} 
gives furthermore 
\[ \lan d a_{1}, da_{1} \ran =  \lan d a_{1}, \frac{1}{2}\langle |\phi |^2 \rangle -\frac{1}{2}|\phi |^2 \ran\]\[ = - \lan d a_{1}, \frac{1}{2}|\phi |^2 \ran = - \lan \frac{1}{2}\langle |\phi |^2 \rangle -\frac{1}{2}|\phi |^2, \frac{1}{2}|\phi |^2 \ran\]
Thus we conclude that 
 \begin{align} \label{da1-norm}  \frac{1}{|X|}\|da_{1} \|^2 = - \frac{1}{4}\langle |\phi |^2 \rangle^2 + \frac{1}{4}\lan |\phi |^4 \ran.\end{align}
The last two relations imply
\[  \frac{1}{|X|}\lan \phi,  2i a_{1}\cdot\n_{a^{n}}\phi \ran =  \frac{1}{2}\langle |\phi |^2 \rangle^2 - \frac{1}{2}\lan |\phi |^4 \ran.\]
 This equation together with \eqref{solvcond'}, {\bf the relations $\lan \phi,  \phi \ran=|X|\langle |\phi |^2 \rangle$ and $ \kappa^2\lan \phi,  |\phi|^2 \phi\ran=|X|\langle |\phi |^4 \rangle$} and the definition \eqref{beta} gives \eqref{lam1}. \end{proof}

Eqn \eqref{lam1} fixes the parameter $s$ uniquely up to the normalization of $\psi_0$. Indeed, we observe that the third equation in \eqref{epsexpansions} implies   $s^2 = \frac{\mu-b \kappa^2}{\mu_1}+O((\mu- b)^2)$, which, together with 
 \eqref{lam1} and the normalization $\langle |\phi |^2 \rangle=1$, 
 yields  
\begin{equation}    \label{s}
         s^2 
       = \frac{\mu-b \kappa^2}{ (\kappa^2 -  \frac{1}{2})\beta(a^n) +\frac{1}{2}} +O( (\mu-b \kappa^2)^2). 
    \end{equation}
This implies \eqref{s-r}.
  This proves Theorem \ref{thm:GLE-RS-exist-cond}. $\qquad \qquad \Box$
 
 \begin{corollary} \label{cor:leadingorder-unresc}
The unrescaled bifurcating solutions,  $(\psi_s,\ a_s,\ r_s)$, constructed in Theorem \ref{thm:GLE-RS-exist} 
 have the following expansions 
\begin{equation} \label{epsexpansions-unresc}
\psi_{s} = s\phi+ O(s^3),\ a_s =a^n +  s^2 a_{1} +  O(s^4), \ r_{s} = b + s^2 r_1 + O(s^4),  \end{equation}
where  $\phi$ and $a_{1}$ satisfy the equations in \eqref{phi-a1-eqs} 
  and $r_1=\mu_1/ \kappa^2$, with $\mu_1$ given by \eqref{lam1}.
     \end{corollary}

\section{Proof of Theorem \ref{thm:ener-asymp}} \label{sec:en-expan}

   Multiplying the first equation in \eqref{GLE} scalarly by $\psi$ and integrating by parts gives
   \begin{equation*}
  		\lan \nabla_a \psi, \nabla_a \psi\ran = \kappa^2 \int_{X} \left(|\psi|^2 - |\psi|^4\right).
	\end{equation*}
Substituting this into the expression, \eqref{GLen}, for the energy, 
 we find
   \begin{equation}   \label{asymp:Elambda'}
		\mathcal{E} (\psi, a)  = - \frac{\kappa^2}{2} \| |\psi|^2\|^2 + \|d a\|^2+  \frac{\kappa^2}{2} |X|.
            \end{equation}
    Using the expansions  \eqref{epsexpansions-unresc} and the facts that $d a^1 =2\pi  \om$ and $\langle d a_1\rangle= 0$  gives
     \begin{align}
    \label{asymp:a-en}
	   \|d a\|^2 &= \|d a^n\|^2 +2 s^2 \langle d a^n, d a_1 \rangle +s^4\|d a_1\|^2 + O(s^6)\\
	   & = (2\pi)^2 |X| +s^4\|d a_1\|^2 + O(s^6).
            \end{align}
 The last two relations, together with the first equation in \eqref{epsexpansions} and equation \eqref{da1-norm},  and the normalization $\langle |\phi |^2 \rangle=1$ give 
  \begin{equation}     \label{Eexpan-s}
    \frac{1}{|X|}    \mathcal{E} (\psi_{s}, a_{s}) =         \frac{\kappa^2}{2} +(2\pi)^2 
    - \frac12 s^4 [(\kappa^2 -  \frac{1}{2})\beta(a^n) + \frac{1}{2}] 
          +  O(s^6).
           \end{equation}
This together with \eqref{s} implies \eqref{Eexpan}.  
$\qquad \qquad \Box$ 

\subsection*{Acknowledgements}
The first and fourth authors' research is supported in part by NSERC Grant No. NA7901. During the work on the paper, the fourth author enjoyed the support of the NCCR SwissMAP. 
The second author's work is supported by NSF Grant No. DMS 1615921.
The third author acknowledges the University of Saskatchewan for a New Faculty Recruitment Grant.

\appendix

\section{Analysis on the universal cover} \label{sec:equiv-formul}

\subsection{Generalities}  \label{sec:general-cover}

To lift \eqref{GLE} to the universal cover, $\tilde X$, of $X$,  let $\pi_1(X)$  act on $\tilde X$ by deck transformations, whose group is denoted by $\Gamma$, and let   $\rho$ be an automorphy map, i.e. $\rho: \Gamma \times \tilde X \rightarrow \mathbb{C}^*$ and satisfies the co-cycle relation 
\begin{align}\label{co-cycle} \rho(\g \cdot \g', z) = \rho(\g, \g'  z)\rho(\g', z).\end{align} 
Then any holomorphic line bundle, $E$ over $X$ is isomorphic  to one of the form $E_\rho = \tilde E / \rho$, where $\tilde E / \rho$ is a factor space according to the action of $\G$ 
\begin{align}\label{Gam-action} (z, \psi) \ra (\gamma z, \,\, \rho(\gamma, z) \psi),\ \forall \g\in \G.\end{align}
 The Chern class, $c_1(\rho)$, of $\rho$ equals the degree of 
  $E_\rho$ (see \cite{Gun2}, Theorem 2a for the definition of $c_1$). 


Pulling back the section $\psi$ and connection form $ a$ to $\tilde{E}$, we arrive at the section $\Psi$  and connection form $ A$ on  $\tilde{E}$, which satisfy the relations
\begin{equation}\label{gauge-per-Gam}\g^*\Psi =  \rho_\g \Psi,\ \g^*A = A - i \rho_\g^{-1} d \rho_\g ,\ \forall \g\in \pi_1(X), \end{equation}
where  $\g^*$ is the pull back of sections and connections by $\g$ 
and where we have written $\rho(\g, z) \equiv \rho_\g(z) $. 
Conversely, if $\Psi$ and $A$ are a section 
and a connection form  in  $\tilde{E}= \tilde X \times \mathbb{C}$, then $\Psi$ and  ${A}$ project to a section $\psi$ and a connection form $ a$ on $E$ if and only if they satisfy \eqref{gauge-per-Gam}, with the automorphy factor $\rho$ corresponding to $E$.
 Moreover, $F_A$ on $\tilde{E}$ descends to $F_a$ on $E$.

We say that a pair $(\Psi,  A)$ is {\it $\rho-$equivariant}  (or {\it gauge $\pi_1-$invariant}) iff it satisfies \eqref{gauge-per-Gam}  for some  automorphy factor $\rho$. 

For convenience of references we summarize a part of the above discussion as 

\begin{proposition}\label{prop:corresp} There is a one to one correspondence between sections and connections on $E$ and   $\rho-$equivariant sections and connections on $\tilde{E}= \tilde X \times \mathbb{C}$ (i.e $\rho-$equivariant functions and one-forms on $\tilde X$). \end{proposition}

It is convenient to formulate the next property of the correspondence $(\psi, a)\Leftrightarrow (\Psi, A)$ as an elementary proposition:

\begin{proposition}\label{prop:proj-lift-sol} 
The form of equations (\ref{GLE}) does not change when lifted to the universal cover.  $(\psi,a)$ solve $(\ref{GLE})$ on $E$ iff its lift $(\Psi, A)$ to 
$\tilde{E}$ solves $(\ref{GLE})$ on $\tilde E$. 
\end{proposition}

Dealing with $\rho-$equivariant functions and one-forms on $\tilde X$, rather than with sections and connections on the bundle $E$, is convenient because of the global coordinates on $\tilde X$. Using this we give in Appendix \ref{sec:bundle-holomorphiz} a hands-on proof of Theorem \ref{thm:bundle-holomorphiz}.

Recall that, since $X$ is a Riemann surface of the genus $g\ge 2$, its universal cover, $\tilde X$, can be identified with the Poincar\'e upper complex half plane, $\mathbb{H}$, and  $\pi_1(X)$ with a  Fuchsian group $\G$  (acting on $\mathbb{H}$).  Using the standard co-ordinate $z$ on $\bH$, we see that $\pi_1(X)$ is a subgroup of $PGL(2, \R)$ acting on $\mathbb{H}$ by M\"obius transforms, $\gamma  z = \frac{az + b}{cz + d},\  \g=\left(\begin{smallmatrix}a&b\\c&d\end{smallmatrix}\right)  \in \G$. 

 Recall also that $PGL(2, \R)$ is a group of isometries of $\mathbb{H}$ with the standard hyperbolic metrics $\tilde h=  (\im z)^{-2} |dz|^2$. 
The latter generates the hyperbolic  volume form \[\omega = \frac{i}{2}\Im(z)^{-2} dz\wedge d\overline{z}.\] 
 Finally, recall that the character of $\G$ is a homomorphism $\s: \G  \rightarrow U(1)$. We have the following results proven in Sections \ref{sec:af} and \ref{sec:cc-con-exist}. 

\begin{theorem} \label{thm:aut-factsEn}
For any $n\in \Z$, the map $\rho_{n}: PGL(2, \R)\times \bH\ra U(1)$, given by 
 \begin{align}\label{n-automor}
\rho_{n}(\g, z) & = \biggr[\frac{c\overline{z} +d}{cz + d}\biggl] ^{-\frac{n }{2g-2} },\ \text{ for }\  \g=\left(\begin{smallmatrix}a&b\\c&d\end{smallmatrix}\right)  \in PGL(2, \R), 
\end{align}
is an automorphy factor. 
Consequently, for any Fuchsian group $\G$ and any character $\s: \G  \rightarrow U(1)$, 
 the map $\rho_{n, \s}: \G\times \bH\ra U(1)$ given by   
\begin{align}\label{n-sig-automor}
\rho_{n, \s}(\g, z) & = \s(\g) \rho_{n}(\g, z), 
 \text{ for }\  
 \g \in \G, 
\end{align}
is also an automorphy factor. The Chern class of $\rho_{n, \sigma}$ is $c_1(\rho_{n, \sigma})=n$. 
\end{theorem}   

\begin{theorem}\label{thm:conn-const-curv}
For any $n\in \Z$, the connection, $A^n$,  on the trivial bundle $\tilde E := \bH\times \C$, given by
\begin{align}\label{An} A^n = \frac{n }{2g-2} y^{-1}dx,\end{align} 
(a) has a constant curvature with respect to the standard hyperbolic volume form on $\bH$ ($F_{A^n}  = \frac{n }{2g-2} \omega$);   
(b) is equivariant with respect to the automorphy factor \eqref{n-sig-automor} for any Fuchsian group $\G$ and any character $\s: \G  \rightarrow U(1)$; 
 (c)  is  unique up to proper, i.e. with $g=e^{i f},\ f:X\ra \R$,  gauge transformations. \end{theorem}  

The projection of $A^n$ to $E_{n, \s}$ gives the distinguished connection $a^{n, \s}$ on $E_{n, \s}$.

\begin{remark} \label{aut-char}
 As mentioned above one may associate to an automorphy $\rho$ the bundle $E_\rho$. In fact this extends to the non-uniformized setting. The ingredients for this are the characters $\sigma$ on $\G$ and the degree $n$ which we may think of as a character on the abelian group $\mathbb{Z}$. In the general setting the automorphy factors of the uniformized case become unitary characters on the product $\frak{S} = \mathbb{Z} \times \G$. There is a 1:1 correspondence between unitary characters and holomorphic line bundles on $X$ which gives one yet another way to describe the Picard group: 

$Pic^{(n)}(X)$ = group of unitary characters $\rho$ on $\frak{S}$ restricted to $c_1(\rho) = n$. 

Hence, it is natural to label holomorphic line bundles as $E_{n, \sigma}$ where $\sigma$ is a unitary character of $\G$.

Because these characters take their values in an abelian group, $U(1)$, they only depend,  on the abelianization of $\pi_1(X)$ which we will denote by $\G_{\rm abel}$.  As with $\pi_1(X)$ itself,  $\G_{\rm abel}$ depends on the basepoint for the paths. This appears in the explicit realizations of these characters as, for instance, in \eqref{n-automor} which depends on $z$, unless $n = 0$.
 However, abstractly, the abelianization is isomorphic to the first homology group: 
$\G_{\rm abel} \simeq H_1(X, \mathbb{Z})$.
\end{remark}
\begin{remark} \label{centext}
In the more general setting of vector bundles, characters are replaced by representations and when $c_1 \ne 0$ one must consider representations of central extensions of $\pi_1(X)$ \cite{AB}. However, in the case of line bundles where the gauge group is abelian, the central extension splits into a direct product of $\mathbb{Z}$ with $\pi_1(X)$; the main remnant of this extension in the character is the degree $n$.
\end{remark}

\begin{remark}\label{rem:af-gauge-equiv} 
1)  
The description in the above theorem does not depend at all on the complex structure of the underlying Riemann surface and therefore serves as a very useful gauge for analyzing this problem.

\medskip 

2) A section $g$ of the line bundle over $X$ with fibers $Aut (E_x)$ 
 defines the isomorphism 
\begin{align}\label{g-iso-E} g E_\rho = E_{g_*\rho}, \ g_*\rho_{\gamma} := (\g^* \tilde g) \rho_{ \gamma}\tilde g^{-1},\end{align}
 where  $\tilde g$ is a lift of a section $g$ and we used the notation $\rho_\gamma(z)=\rho(\gamma, z)$. 
 This gives the `gauge' transformations of the automorphy factors.  Thus, changing a connection using gauge equivalence changes the bundle automorphy factor, so, in general, it is impossible to adjust both simultaneously.

\medskip
 
Consider a special case of \eqref{g-iso-E}:

If $g$  is such that its lifting $\tilde g$ to $\tilde X$  ($g \circ \pi =\tilde g$) satisfies 
 \[\tilde g (\g z) = \tilde g (z)\s'(\g)$, $\s' \in Hom(\pi_1(X), U(1)),\] then  $g_*\rho_{n, \s} = \rho_{n, \s'\s} $. 
In particular, if $g:X\ra U(1)$ (i.e. it is a section of a trivial bundle, so that $\tilde g$ satisfies $\tilde g(\g x)=\tilde g(x),\ \forall \g\in \pi_1(X)$), then $g_*\rho = \rho $. 

\medskip 
2) The gauge invariance implies that we can consider GLEs on a fixed bundle $E_{n, \s}, n\in\Z, \s\in Hom(\pi_1(X), U(1))$. 
 \end{remark}

\subsection{Automorphy Factors}\label{sec:af}

 In this section, we give the calculation proving Theorem \ref{thm:aut-factsEn}.
Recall $\mathbb{H} = \{ z \in \mathbb{C}\,|\, \Im(z) > 0\}$ with 
 the group $\mathbb{P}GL(2,\mathbb{R})$ acting on $\bH$ by $\left(\begin{smallmatrix}a&b\\c&d\end{smallmatrix}\right) \cdot z = \frac{az + b}{cz +d}$.  
\DETAILS{We have 
\begin{theorem}\label{thm:aut-factsEn}
For any Fuchsian group $\G$, any $n\in \Z$ and any  character $\s: \G   \rightarrow U(1)$, the map $\rho: \G \rightarrow \mathbb{C}^*$, given by 
 \begin{align}\label{b-automor}
\rho(s, z) = \s(s) 
 \biggr[\frac{c\overline{z} +d}{cz + d}\biggl] ^{-\frac{n }{2g-2} },\ \text{ for }\  & s=\left(\begin{smallmatrix}a&b\\c&d\end{smallmatrix}\right)  \in \G,\notag\\ &  s(z) =  \frac{az + b}{cz +d}.
\end{align} 
 is an automorphy factor; i.e., it satisfies the co-cycle condition \eqref{co-cycle}.  Its Chern number is $n$.
\end{theorem}}
\begin{proof}[Proof of Theorem \ref{thm:aut-factsEn}] Let $\beta=\frac{n }{2g-2}$ and $s=\left(\begin{smallmatrix}a&b\\c&d\end{smallmatrix}\right),\ t=\left(\begin{smallmatrix}e&f\\g&h\end{smallmatrix}\right)  \in \G$. Using \eqref{n-automor}, we compute 
\begin{align*}
\rho_n(s \cdot t, z) &= 
\biggr[\frac{(ce + dg)\overline{z} +(cf + dh)}{(ce +dg)z + (cf + dh)}\biggl] ^{-\beta} \\
&= 
 \biggr[\frac{c(e\overline{z} + f)  +d(g\overline{z} + h) }{(c(ez + f) + d(gz + h)}\biggl] ^{-\beta} \\
&= 
 \biggr[\frac{c\frac{e\overline{z} + f}{g\overline{z} + h} +d}{c\frac{ez + f}{gz + h} + d}\biggl]^{-\beta}  \biggr[\frac{g\overline{z} + h}{gz + h}\biggl]  ^{-\beta} \\
&= \rho_n(s, t \cdot z)\rho_n(t,z).
\end{align*}
Using the formula for  the Chern class, $c_1(\rho)$, of  a co-cycle $\rho$ (see \cite{Gun2}, Theorem 2a), we compute $c_1(\rho)=n$. 
\medskip
\end{proof}

\subsection{Uniform constant curvature connection and its holomorphic structure}  \label{sec:cc-con-exist}

In this section we prove Theorem \ref{thm:conn-const-curv}.
We begin with some preliminary constructions. 

 We consider the trivial bundle, $\tilde E:= \bH \times \mathbb{C}$ with the standard complex structure on $\bH$ associated to the  hyperbolic metric $\tilde h=  (\im z)^{-2} |dz|^2$. 

 Since we work here on a global product space, it is natural to take the fiber metric to be induced from the metric on the base. So we take the metric on the fiber $\mathbb{C}$ over the point $z \in \bH$ to be $k_z  = (\im z)^{-2} |dw|^2$ where $w$ is the coordinate on the fiber $\mathbb{C}_z$.

Let the connection $A$ be given by $A:=A_1 d x_1+A_2d x_2$. We decompose the covariant derivative  $\nabla_A$ into $(1, 0)$ and $(0, 1)$ parts as $\nabla_A=  \partial_A' +  {\partial_A''}$, where
  \begin{align} \label{pA}\partial_A' := \partial + i A_c,\  \partial_A'' :=\overline{\partial} + i \bar A_c.\end{align} 
  Here $\p := \frac{\partial}{\p z}\otimes d z$ and $\overline{\partial} := \overline{\frac{\partial}{\p z}}\otimes d\overline{z}$, where, as usual, $ \frac{\partial}{\p z}:=  \p_{z}:=(\partial_{x_1} - i \partial_{x_2})/2$ and $ \overline{ \frac{\partial}{\p z}} \equiv  \p_{\bar z}:=(\partial_{x_1} + i \partial_{x_2})/2$ and 
\begin{align} \label{compl-conn} A_c:=\frac12 (A_1 - iA_2) \otimes d z,\ \bar A_c:=\frac12 (A_1 + iA_2) \otimes d \bar z. \end{align}
We call the complex one-form $A^c$ the complexification of the real connection $A$.\footnote{The operators $\partial_A'$ and $\partial_A''$ are defined directly as follows. Let $J$ be the natural almost complex structure on $X$ (a linear endomorphism of $T^*\tilde{X}$ satisfying $J^2=-\one$), generated by $J(dx_1) = dx_2$ and $J(dx_2) = -dx_1$. Then $\Pi_\pm : T^*\tilde{X} \rightarrow T_{\tilde{X}}^{*1, 0}/T_{\tilde{X}}^{*0,1}$ via $v \mapsto (v \pm i J(v))/2$ defines the (scaled) projection onto antiholomorphic forms. We define $\partial_A' = \Pi_+(\nabla_A)$ and $\partial_A'' = \Pi_-(\nabla_A)$ which is computed explicitly to be as in \eqref{pA}.}
In the reverse direction, we have $A = 2\re A_c$. 

In terms of $A_c$, the curvature is given by   $F_A=2 \re \bar \p A_c$. Moreover, if $A_c$ satisfies the equivariance relation $s^* \bar A_c = \bar A_c  - i \overline{\partial}\tilde f_s$, then $A$ satisfies $s^* A = A  + d f_s$, with $f_s$ satisfying $d f_s:= 2 \im \overline{\partial}\tilde f_s$.

  According to \eqref{compl-conn}, the  complexification of the connection $A^n$ given in the theorem   is $A_c^n = \frac{n }{2g-2} \frac{ 1}{2\im(z)}d z$. 

\begin{proof}[Proof of Theorem \ref{thm:conn-const-curv}] 
  We omit the superindex $n$ in $A^n$ and $A_c^n$ and use the notation $b=\frac{n }{2g-2}$. %

\paragraph{Proof of constant curvature.} Using that $\omega = \frac{i}{2}\Im(z)^{-2} dz\wedge d\overline{z}$, we find
\begin{align*}
\bar \p  A_c &= \frac{\partial}{\partial \bar z} \frac{i b}{z - \overline{z}}d\bar z\wedge d z 
=  \frac{- i b}{(z - \overline{z})^2}dz\wedge d\overline{z} = \frac{i b}{4\Im(z)^2}dz\wedge d\overline{z} = \frac{ b}{2}\omega.
\end{align*}
Since $F_A= 2 \re \bar \p A_c$, this gives the desired result. 
\paragraph{Proof of uniqueness.}

If $A$ and $B$, satisfy $dB = dA$ then $d(A-B) = 0$. It follows from the simple connectedness of $\mathbb{H}$ 
 that $A-B = df$ for some function $f: \mathbb{H} \rightarrow \mathbb{R}$ and $f$ is unique up to an additive constant.  So we can map $B$ to $A$ through a suitable retrivialization.  
This completes the proof.

\paragraph{Proof of equivariance.}  For 
a generic isometry $s(z) =  \frac{\al z + \bet}{\g z +\del}$, we have $\frac{\partial s(z)}{\partial z}=(\g z +\del)^{-2}$ and $\Im(s(z))=\frac{\Im(z)}{ |\g z + \del|^2} $, which gives 
\begin{align*}
s^* \bar A_c &= \frac{ b}{2\Im(s(z))}\overline{\frac{\partial s(z)}{\partial z}}d\overline{z}= \frac{ k |\g z + \del|^2}{2\Im(z)} \frac{d\overline{z}}{(\g \overline{z} + \del)^2}  \\
&= \frac{ b}{2\Im(z)}\frac{\g z + \del}{\g \overline{z} + \del}d\overline{z}= \frac{ b}{2\Im(z)}d\overline{z} + \frac{ b (\g z - \g \overline{z})}{2\Im(z)(\g \overline{z} + \del)} d\overline{z} \\
& =\bar  A_c + \frac{i b \g}{\g \overline{z} +\del} d\overline{z}  =: \bar A_c  +  \overline{\partial}\tilde f_s,\end{align*} 
where $\tilde f_s$ is the function defined by the last relation, i.e. $ \overline{\partial}\tilde f_s=\frac{i b \g }{\g \overline{z} +\del}d\overline{z}.$
  Solving this equation, we find \[\tilde{f_s} = b\text{ln}(\g \overline{z} + \del)+c_s.\] 
  Now, we define $f_s := 2\re \tilde{f_s}$ and use that $\re (\bar\p\tilde{f_s}) =d (\re \tilde{f_s})$ (as can be checked by the direct computation: $\frac1b \re (\bar\p\tilde{f_s}) = - \frac{\g^2 x_2}{|\g z +\del|^2} dx_1 + \frac{\g (\g x_1 +\del)}{|\g z +\del|^2} d x_2$) to obtain $s^* A =  A  + d f_s$, with
\begin{align*}
& f_s = 2\Re(\tilde{f_s}) = i b \text{ln}\biggr[\frac{\g \overline{z} +\del}{\g z + \del}\biggl] +c_s,\\  \numberthis \label{gs} \
&\rho(s, z) = e^{if_s(z)} =e^{i c_s}  \biggr[\frac{\g \overline{z} +\del}{\g z + \del}\biggl] ^{-b}.
\end{align*}
Since $\mathbb{H}$ is the upper half plane, the complex logarithm is well defined and $\g z + d$ is always non zero. \end{proof}

\begin{remark} 1) The function $\tilde f_s (z)$ appearing above gives the character  $\tilde \rho (z) = e^{\tilde f_s (z)}$, which  is now $\mathbb{C}^*$ valued instead of $U(1)$ valued. 

2)
$A^n$ is $\mathbb{R}$ - linear while $A^c$ is naturally $\mathbb{C}$ - linear. It is natural to ask what the action of $i$ on $A^n_c$ does when we map back to $A^n$. A simple calculation shows that $iA^n_c = 
i \frac b2 \frac{ 1}{\im(z)}d z$ maps to $b y^{-1}dy$ which is flat! It turns out that the complex action of $i$ induces a rotation into the space of flat connections. 
\end{remark}

Since $(\partial_{A^n}'')^2 = 0$, the partial connection $\partial_{A^n}''$  gives a holomorphic structure on $E_n$, unique up to a complex gauge transformation, which in turn corresponds to the character \eqref{n-sig-automor} (see Remark \ref{aut-char}).

\begin{proposition}\cite{Gun4} 
For $b \in \mathbb{Z}$ the holomorphic sections of $H^0(X, E \otimes F)$ are modular forms of weight $-b$ for $\G$.  (See Appendix \ref{sec:bundleEF} for details.)
\end{proposition}

\begin{remark}
 For non-integral $b$, the sections of the proposition lift to modular forms of weight $-\frac{n}{\gcd(n, 2g-2)}$ on a $\frac{2g-2}{\gcd(n, 2g-2)}$ cover of $X$  (\cite{FarKra} III.9.13).
\end{remark}

\section{Variation of the hermitian metric} 

 Given a complex structure, the Riemannian metric can be specified in local complex coordinate $z$ by a positive smooth function $\lambda(z)$ as  $ds^2 = \lambda(z) |dz^2|$ and the corresponding volume form by $\omega = \lambda(z) dz \wedge d\bar{z}$. Solutions to \eqref{GLE} depend parametrically on $h$ only through the positive, smooth function $\lambda(z)$. In this appendix we first consider how our variational equations change if 
$\lambda(z)$ is simply rescaled. Then we describe what happens under more general deformaitons. 

\subsection{Rescaled equations} \label{sec:rescGLE} 
Here we show that the Ginzburg-Landau equations, \eqref{GLE}, with 
the hermitian metric $r h$ are equivalent to  \eqref{GLE'}, with $\mu=r \kappa^2$.  We distinguish the quantities related to the  family of Hermitian  metrics $d s= r \lam (z) |dz|$ by tildes. It suffices to show that
\begin{align}\label{scaling}
\tilde{\Delta}_A  
= \frac{1}{r} \Delta_A,\  
\tilde{M}  =  \frac{1}{r} M, 
\end{align}
where, recall, $M:=d^*d $. 
We note that $-\Delta_A= - *\overline{(- d + iA)}*(d + iA)$, 
where $*$ is  the Hodge star operator and that the scaling of the metric arises only in the application of the Hodge star.
To prove the first relation, we use that, since $g_{ij}$ is conformally flat and therefore $dx_1$ and $dx_2$ are orthogonal, we have that 
\begin{equation}\label{star}
*dx_1  = dx_2,\  *dx_2 = -dx_1,\ *1 = \lambda dx_1 \wedge dx_2,\ *dx_1 \wedge dx_2 = 1/\lambda. \end{equation}
Indeed, $*dx_1= \frac{||dx_1||}{||dx_2||} dx_2    = \frac{\sqrt{g^{11}}}{\sqrt{g^{22}}}dx_2 = \frac{\sqrt{1/\lambda}}{\sqrt{1/\lambda}}dx_2 = dx_2$ and similarly for the other relations. This gives $\tilde{\Delta}_A  = \tilde * \overline{(d + iA)}\tilde *(d + iA) =  \frac{1}{r}\ \overline{(d + iA)} *(d+ iA)= \frac{1}{r} \Delta_A.$ We note that taking $\tilde{\psi} = \sqrt{r}\psi$ produces the rescaled equations, which after omitting the tilde read \eqref{GLE'} with $ \mu:=\kappa^2 r$ (cf. \cite{Bradlow}).

\subsection{Gauged Deformations of the metric} \label{sec:bundleEF}
We fix a base volume form $\lambda_0(z) dz \wedge d\bar{z}$ which determines a fixed base * operator. Then any other * operator gotten from the base volume form, or equivalently the base * operator, by multiplying with the factor $\lambda(z)/\lambda_0(z)$. To ensure the non-vanishing of the volume form we set $\lambda = e^{h(z)}$ and vary $\lambda(z)$ by varying $h$. Now we make some convenient definitions:
\beann
\lambda_0(z) &=& e^{h_0(z)}\\
\lambda(z)/\lambda_0(z) &=& e^{h(z) - h_0(z)}\\
g(z) &=& e^{-(h(z) - h_0(z))}
\eeann 
Finally we let $\widehat{*}$ denote the * operator associated to $\lambda_0$ while $*$ denotes the * operator associated to $\lambda$. Now note how the three terms in the energy \eqref{GLen} scale with $\lambda(z)$. The first term is independent, the second term scales as $\lambda^{-1}(z)$ and the third term scales as $\lambda(z)$. Now, integrating by parts to get the variational equations a piece of the first term combines with the third term to get the first variational equation in \eqref{GLE} while the remaining piece of the first term combines with the second term to yield the second variational. It is this mixing of $\lambda$ scalings between terms that determines how varying the metric affects the variationa equations. Indeed, the following calculations are straightforward. 
\beann
d^{\widehat{*}} d a &=& g\left( (d + g^{-1}dg)^* d a\right) = \Im(\overline{\psi}\nabla_a \psi)\\
-\widehat{\Delta}_a\psi &=& - g \Delta_a\psi = (\kappa^2|\psi|^2 - 1) \psi.
\eeann
Now if one makes the gauged scaling $\psi \to \sqrt{g}\psi$ one sees that varying $h_0$ by $h - h_0$
amounts to transforming the original variational equations to
\beann
(d + g^{-1}dg)^* d a &=& \Im(\overline{\psi}\nabla_{a + \frac12 g^{-1}dg} \psi)\\
-\Delta_{a + \frac12 g^{-1}dg}\psi &=& (\kappa^2|\psi|^2 - g^{-1}) \psi.
\eeann
We note that this is effectively a gauge transformation; however, one which is real-valued and non-unitary. As a consequence the unit 1 in the nonlinear term of the final equations must replaced by $g^{-1}$. This is completely analogous to, and indeed is an extension of, the constant rescaling carried out in Appendix  \ref{sec:rescGLE}. We may view this gauging as a smoothly varying pointwise dilation which is completely consistent with the geometric significance of $\lambda$ in our description of surface metrics below.

\section{Bundle Holomorphisation} \label{sec:bundle-holomorphiz}
In this section we prove Theorem \ref{thm:bundle-holomorphiz}. 
Let $E$ be a $U(1)$  bundle over $\mathbb{H}$ with an automorphy factor 
$\rho$  and a unitary, $\rho-$equivariant connection $\nabla_A = \nabla + iA$. 

The following result is a consequence of the Dolbeault lemma (see \cite{Forster}, Theorem 13.2): 

\begin{proposition} 
\label{prop:dolbeault}There exists a complex valued  gauge transformation $g$ that solves $g^{-1} \partial_A'' g = \overline{\partial}$ and it is unique up to a holomorphic term. In particular, $g$ satisfies 
\begin{equation}
\overline{\partial} g = i g \overline{A_c}. \label{dolbeault}
\end{equation}
 $g$ inherits an equivariance property from $A$.
\end{proposition}
\begin{proof}
Define $g = e^\kappa$, then $\overline{\partial}\kappa = i\overline{A_c}$. 
 Passing from forms to functions and using the Dolbeault lemma (see \cite{Forster}, Theorem 13.2), we arrive at the statement of the propositions. 
\end{proof}

\begin{proposition}\label{prop:rho'}Viewing $g$ as a multiplicative map from $\tilde E$ to a new bundle $\tilde E'$, the induced automorphy factors 
 $\rho' (\gamma):=\gamma^*(g^{-1})\rho(\gamma) g$ on the new bundle are holomorphic.
\end{proposition}
\begin{proof} We start with two key components about how holomorphic functions compose with regular complex valued functions. If $w: U \rightarrow V$ is holomorphic in a neighbourhood, $x : \mathbb{R} \rightarrow U$ and $g : \mathbb{C} \rightarrow V$ are smooth then
\begin{align*}
\frac{d w \circ x}{dt} &= \frac{dw}{dz}(x(t))\frac{dx}{dt} \\
\overline{\partial}(g \circ w) &= (\overline{\partial} g) (w(z))\overline{\frac{dw}{dz}}
\end{align*}
which can be checked by regarding maps on $\mathbb{C}$ as maps on $\mathbb{R}^2$. With these two formulas we can perform the computation:
\begin{align}
\overline{\partial} &(\gamma^*(g^{-1})\rho(\gamma) g) \\
=  &\overline{\partial} (\gamma^*(g^{-1}))\rho(\gamma) g + \gamma^*(g^{-1}) \overline{\partial}(\rho(\gamma)) g + \gamma^*(g^{-1})\rho(\gamma)\overline{\partial}g \\
= & i \gamma^*(g^{-1}) \rho(\gamma) g (\overline{A_c} + \overline{\partial}f_\gamma) + \overline{\partial} (1/z \circ g \circ \gamma) \rho(\gamma) f \\
=  & i \gamma^*(g^{-1}) \rho(\gamma) g (\overline{A_c} + \overline{\partial}f_\gamma) \\
 & -1/z^2 \circ g \circ \gamma \cdot \overline{\partial}g \circ \gamma \cdot \overline{\partial \gamma} \\
= & i \gamma^*(g^{-1}) \rho(\gamma) g (\overline{A_c} + \overline{\partial}f_\gamma) -  i \gamma^*(g^{-1}) \rho(\gamma) g \gamma^* \overline{A_c} \\
= & i \gamma^*(g^{-1}) \rho(\gamma) g (\overline{A_c} + \overline{\partial}f_\gamma -  \gamma^* \overline{A_c}) \\
= & 0.
\end{align}

The first equality is the Leibniz rule, the second equality follows from ($\ref{dolbeault}$), the third equality is an application of the two formulas and the last equality follows from the equivariance condition on $A$. 
\end{proof}
Now we derive Theorem \ref{thm:bundle-holomorphiz} from these propositions. Let $\tilde E$ be the uniformization of the unitary bundle $E$ in Theorem \ref{thm:bundle-holomorphiz}  and let $\n_A$ be the lift of the connection, $\n_a$, on that bundle. Hence $E=\tilde E/\rho$ for some automorphy factor $\rho$. Furthermore, let $\tilde E', \rho'$ and $\n_{A'}$ be the bundle, the automprphy factor and the connection constructed in Propositions \ref{prop:dolbeault} and \ref{prop:rho'} and let $E'=\tilde E'/\rho'$. 

Since the function $g$ in Proposition \ref{prop:dolbeault} is equivariant, it descends to $X$ as a section of a line bundle over $X$, which induces the map of $E$ into $E'$. This map takes also  
 the connection $\n_a$ on $E$ into the connection $\n_{a'}$ on $E'$. By Propositions \ref{prop:dolbeault} and \ref{prop:rho'}, $\n_{a'}$ comes from $d$ on $\tilde E'$ via a holomorphic projection. Hence the connection $\n_{a'}$ is holomorphic.  
This proves 
Theorem \ref{thm:bundle-holomorphiz}. $\Box$

\bigskip

For connection \eqref{An} (with the standard metric on $\bH$), we have an explicit form the transformation $g$, which we now denote $\tilde g$. Namely, we have

\begin{proposition}\label{prop:DeltaA-repr}
Let $b := \frac{2\pi n }{|X|}= \frac{n }{2g-2}$. Then  $  \tilde g = y^b$ solves $\tilde g^{-1}\partial_{A^n}''   \tilde g = \overline{\partial}$. $ \tilde g$ transforms under $s=\left(\begin{smallmatrix}a&b\\c&d\end{smallmatrix}\right)  \in PSL(2,\mathbb{R})$ as $s^*  \tilde g =  \tilde g \tilde{\rho},$ where
\begin{equation}\label{tilde-automor}
\tilde{\rho}: PSL(2,\mathbb{R})\times \bH \ra \R,\  \tilde{\rho}(s,z):= |cz+d|^{-2 b}. \end{equation} \end{proposition} 
\begin{proof} We omit the superindex $n$, let $\tilde g = e^{\al}$ and solve: 
\begin{align*}
& e^{-\al}\partial_{A}''e^{\al} = \overline{\partial} \Longleftrightarrow
 e^{-\al}(\overline{\partial} + i \bar A_c)e^{\al} = \overline{\partial}  \Longleftrightarrow \overline{\partial}\al =  i \bar A_c
\end{align*}
 Since $A_c= \frac{n }{2g-2} \frac{ 1}{2\im(z)}d z$, it then follows that 
\begin{align*} 
\al= b \text{ln}\biggl(\frac{z - \overline{z}}{2i}\biggr) 
\end{align*}
solves $\overline{\partial}\al =  i \bar A_c$, which gives that $\tilde g = y^b$. We compute the automorphy of $\tilde g$ to find \eqref{tilde-automor}.
\end{proof} 

\bigskip
 Let  $F$ be the line bundle over $X=\mathbb{H}/\G$, defined by the automorphy map $\tilde{\rho}^{-1}$, where $\tilde{\rho}$ is given in \eqref{tilde-automor},  and, recall, that $E$ is the line bundle over $X$, with the automorphy map $\rho$. Then $\tilde g^{-1}$ decends to a section, $g$, of  $F$.  It follows that 
  $E\otimes F$ is a holomorphic line bundle and the equation $\partial_{a^n}''\phi = 0$ is equivalent to $g^{-1}\phi$ being a holomorphic section of $E\otimes F$. 
Hence, we have 

 \begin{corollary}\label{cor:NullDeltaA-NullPart} Let $b := \frac{2\pi n }{|X|}= \frac{n }{2g-2}$ and, as above, $s$ be the section of the line bundle $F$ coming from the equivariant function $\tilde g = y^b$ on $\bH$. We have 
   \begin{align}\label{NullDeltaA-NullPart} 
   \Null \partial_{a^n}''  = s H^0(X, E\otimes F), 
  \end{align} 
 where, as usual, $H^0(X, L)$ is the space of holomorphic sections on  a bundle $L$. \end{corollary}

\section{Weitzenb\"ock-type formula} \label{sec:Weitz-form}
We prove Proposition \ref{prop:Delta_a-repr} by passing to the universal cover, $\bH$, and proving there an equivalent relation.  We use the complex covariant derivatives as in the beginning of Section \ref{sec:cc-con-exist}.  
 \begin{proposition}\label{prop:Delta_A-repr}
 \begin{align*}
{\partial''_A}^*{\partial''_A} &= \frac{1}{2}(-\Delta_A - *F_A). \numberthis 
\label{Lapl-formula}
\end{align*} \end{proposition} 
\begin{proof} To prove this we need $*d\overline{z} = *(dx_1 - idx_2) = dx_2 + i dx_1 = id\overline{z}$ and $dz\wedge d\overline{z} = -2i dx_1 \wedge dx_2$. Furthermore we compute the adjoint of ${\partial''_A}$ to be ${\partial''_A}^* = *(-\partial_z \otimes dz - \frac12 (A_2 + i A_1) \otimes dz)*$. Now, we have 
\begin{align*}
{\partial_A''}^*{\partial_A''} &=\frac{i}{4} * (- 2\partial_z \otimes dz - (A_2 + i A_1) \otimes dz) (2\partial_{\overline{z}} \otimes d\overline{z} - (A_2 - iA_1) \otimes d\overline{z}) \\
&=\frac{i}{4}*(-4\partial_z\partial_{\overline{z}} + A_1^2 + A_2^2 - 2(A_2 + i A_2)\partial_{\overline{z}} + 2\partial_z (A_2- iA_1))dz \wedge d\overline{z} \\
&=\frac{i}{4} * (- \Delta +  A_1^2 + A_2^2  -\frac{d}{dx_2} A_1 + \frac{d}{dx_1} A_2 - i A\cdot \nabla - i \nabla \cdot A ))dz \wedge d\overline{z} \\
&= \frac{1}{2\lambda}(- \Delta +  A_1^2 + A_2^2  -\frac{d}{dx_2} A_1 + \frac{d}{dx_1} A_2 - i A\cdot \nabla - i \nabla \cdot A), 
\end{align*}
which gives \eqref{Lapl-formula}. \end{proof}

 \begin{corollary}\label{cor:Delta_A-part_A} Let $A^n$ be a connection of the constant curvature, i.e.  
 $*F_{A^n}=b := \frac{2\pi n }{|X|}= \frac{n }{2g-2}$. 
 Then 
  \begin{align}\label{DeltaA-partA}\Null (-\Delta_{a^n} - b) = \Null {\partial''_{a^n}}
   \end{align} 
Hence,  $b$ is an eigenvalue of 
 $-\Delta_{A^n}$ iff $\Null {\partial''_{A^n}}\neq \{0\}$ and if $b$ is an eigenvalue of  $-\Delta_{A^n}$, then it is the smallest eigenvalue. \end{corollary}

\section{Admissible  connections} \label{sec:admissible-conn}

The purpose of this appendix is to outline relevant parts of the theory of holomorphic line bundles over Riemann surfaces - specifically, the existence of holomorphic sections of such bundles - to help the motivated reader, with the background elsewhere, to navigate the main body of this paper.

 First, recall that holomorphic structures on $E_n$ are typically specified in terms of transition functions for $E_n$ that are holomorphic with respect to the complex structure on $X$.  
 
For our purpose 
it is natural to use another description of holomorphic bundles that is phrased directly in terms of objects we use, namely, 
 derivations (connections), $\p_E''$, of  type $(0,1)$ on $E$; i.e., operators satisfying 
\begin{eqnarray}
\p_E'': \mathcal{E}(X, E) \to \mathcal{E}^{(0,1)}(X, E) \\
\label{deriv} \p_E'' f \xi = \bar{\p}f \cdot \xi + f \p_E'' \xi
\end{eqnarray}
where  $\mathcal{E}(X, E)$ and $\mathcal{E}^{(0,1)}(X, E)$ are the spaces of sections and $(0,1)-$ forms on $X$ with values in $E$, respectively, $f$ is a function on $X$ and $\xi$ is a section of $E$.  
Two derivations are equivalent if and only if they are conjugate to one another under a complex-valued gauge transformation.  

Let 
$\mathcal{A}_{cc}$ denote the space of constant curvature unitary (with respect to $k$) connections on $E$ and let $\mathcal{G}$ denote the group of gauge transformations which preserve $k$. 
\begin{theorem}  \label{thm:KM} (i)  The space of holomorphic structures on $E$ is in a 1:1 correspondence with the space, $\mathcal{C}$, of derivations, $\p''_E$ of  type $(0,1)$ on $E$.   This correspondence descends to the corresponding gauge-equivalence classes. 

(ii)  Let  $\partial_a^{\prime\prime}$ denote the type $(0,1)$-component of $\nabla_a$ and  $\mathcal{G}^c$ denote the group of complex-valued gauge transformations. Then
\begin{eqnarray} \label{NS}
\mathcal{A}_{cc}/\mathcal{G} &\simeq & \mathcal{C}/\mathcal{G}^c,\\ \label{NSmap}
\left[\nabla_a\right] & \to & \left[\partial_a^{\prime\prime}\right]
\end{eqnarray}
where the brackets denote the corresponding gauge equivalence class in each case.
\end{theorem} 
 For the proof of the first statement, see \cite{Kob}, Propositions 1.3.5, 1.3.7 and 1.4.17. We comment on it 
 in Remark \ref{rem:Pic}.   The second one is a special case of a more general result for vector bundles due to Narasimhan and Seshadri that can be found in  \cite{Don1} and  Appendix by O. Garc{\'i}a-Prada in \cite{Wells} (see  Theorem 2.7).

\medskip

Thus equivalent derivations correspond to equivalent holomorphic structures. Given this we will henceforth refer to $\mathcal{C}$ as the space of {\it holomorphic structures} on $E$.

 \begin{remark} \label{rem:Pic}
Given a holomorphic structure on a line bundle $E$, in the sense of holomorphic transition functions on the bundle, and a hermitian metric, $k'$ on $E$  there is a canonical connection, $D_{k^\prime}$, on $E$ which is compatible with the metric $k^\prime$ and whose $(0,1)$-component annihilates holomorphic sections.  (For details we refer the reader to Theorem III.2.1 of \cite{Wells}). This gives the forward direction of Theorem \ref{thm:KM}(i).
\end{remark}

Next, we give some key definitions. $Pic^{(n)}(X)$ denotes 
 the moduli space (i.e. the space of complex gauge equivalence classes ) of  holomorphic line bundles of fixed degree $n$.  
For each $n$, $Pic^{(n)}(X)$ is isomorphic to $Jac(X)$. 
This isomorphism is effectively determined by the Abel-Jacobi map (see \eqref{AJ-map}).

As we mentioned above, we are interested in existence of holomorphic sections of holomorphic line bundles. Though every holomorphic bundle has a meromorphic section, not every such bundle has a holomorphic one.  Let $\Sigma^{(n)} \subset Pic^{(n)}(X)$ denote the subset of degree $n$ holomorphic line bundles which have a holomorphic section. 

 Recall that the divisor, $D$, on a Riemann surface $X$ is a finite collection of points, $P_i\in X$, with associated integers, $n_i$, written as $D:=\sum n_i P_i$. The number $\deg(D):=\sum n_i $ is called the degree of $D$. 

With every meromorphic section, $\vphi$, one can associate the divisor, denoted as $(\vphi)$, which is the collection of its zeros and poles, together with their positive and negative orders (multiplicities), respectively. 
Different sections of a holomorphic line bundle have linearly equivalent divisors: Two divisors are said to be  linearly equivalent if their difference is a divisor of a meromorphic function on $X$. In this way with each line bundle we associate a collection of linearly equivalent divisors.  
Conversely, a divisor uniquely determines an equivalence class of holomorphic line bundles; i.e., a point in $Pic(X)$. (\cite{GH}).


It is a consequence of the definition of the degree of a line bundle in terms of zeros of its sections that the degree of the divisor is equal to the degree of the bundle.   

By this description, the divisor of a holomorphic section has only positive integer coefficients (the value of the coefficient corresponds to the multipicity of the divisor at that point). Such divisors are said to be {\it effective}. An effective divisor corresponds to a unique point in  
 the $n$-fold symmetric product, $X^{(n)}$, which parametrizes unordered $n$-tuples of points on $X$. (It is straightforward to check that $X^{(n)}$ is a smooth manifold \cite{GH}.) 
 
Let $\tilde{\Sigma}^{(n)}$ denote  the set of pairs $(E,D)$ where $E$ is a holomorphic line bundle of degree $n$, which has a holomorphic section, and $D$ is the divisor of a holomorphic section of $E$. Then projection onto the second factor,
$(E,D) \to D$, defines a 1:1 map \begin{eqnarray} \label{pi2}\pi_2: \tilde{\Sigma}^{(n)} \to X^{(n)},\end{eqnarray} which is in fact onto, since one can construct a bundle and its holomorphic section directly from an effective divisor \cite{GH}. One then makes use of a special case of {\it Abel's theorem}:
\begin{theorem} \cite{GH} \label{thm:AT}
Two effective divisors are the zeroes of two holomorphic sections of the {\it same} bundle if and only if they are mapped to the same point in 
\begin{eqnarray} \label{Wn}
W_n := \Phi(X^{(n)})\subset Jac(X),
\end{eqnarray}
Moreover, if 
$D, D'\in\Phi^{-1}(z)$ for $z \in W_n$, then $D$ and $D'$ are linearly equivalent, i.e. there exists a meromorphic function, $f$ on $X$ such that $(f)=D-D'$. 
\end{theorem}
 It also follows from the first statement in Abel's theorem that the admissible bundles,  i.e. the bundles with a dimension one space of holomorphic sections, of degree $n$ correspond to those points $z \in W_n$ for which $\Phi^{-1}(z)$ is a unique point in $X^{(n)}$. 
 \medskip
 
 We can now more properly define the map $I: X^{(n)}  \ra \cA_{cc, n}$ by 
 \[I := \mathcal{I} \circ \pi_1 \circ \pi_2^{-1},\] where $\mathcal{I}: \mathcal{C}/\mathcal{G}^c \ra \mathcal{A}_{cc}/\mathcal{G}$ is the reverse direction of the Narasimahn-Seshadri isomorphism \eqref{NSmap}, $\pi_1: \tilde{\Sigma}^{(n)}\ra Pic^{(n)}(X)$ is given by projection onto the first factor $(E,D) \to E$ in $\tilde{\Sigma}^{(n)}$ and $\pi_2$ is given in \eqref{pi2}.
 (In defining the composition with $\mathcal{I}$ here we are using the isomorphism, $Pic^{(n)}(X)\simeq  \mathcal{C}/\mathcal{G}^c$, which is due to Theorem \ref{thm:KM} (i).

\medskip
 
In fact, one can say a bit more here: under $I$ the admissible connections in 
$\mathcal{I}(\Sigma^{(n)})$ correspond to the regular values in $W_n$ of $\Phi$ (see \cite{FarKra}, III.11.11). By Sard's theorem, this is an open dense submanifold of $W_n$. 

\medskip
Equivalently, by the implicit function theorem, these are the smooth points in the variety $W_n$. So we denote this set by $W_{n, \text{smooth}}$. 
Then $S^{(n)} := \Phi^{-1}(W_{n, \text{smooth}})$ 
  must also be an open dense submanifold of $X^{(n)}$. 

\medskip

\DETAILS{The set of admissible $a_c$, of degree $n$, has the structure of a complex manifold which can be described in terms of $W_n$, 
which is an analytic (in fact algebraic) subvariety of $Jac(X)$. The regular values of $\Phi$ in $W_n$ are precisely the smooth points, $W_{n, \text{smooth}}$, of this analytic variety. The singular points of $W_n$ correspond precisely to the holomorphic structures on $E$ for which $h^0(X,E) > 1$. Again this follows from Abel's theorem \cite{FarKra}, III.11.12. So we have

\begin{corollary}
The space of solutions of (\ref{GLE}), for $1 \leq n \leq g$, whose existence is established in Theorem \ref{thm:AL-exist-RSgener'}, is fibered over the complex sub-manifold $W_{n, \text{smooth}} \subset Jac(X)$. 
\end{corollary}}

\paragraph{Explicit form of $W_{n, \text{smooth}}$.} For certain degrees one can give explicit equations that determine $W_{n, \text{smooth}}$. These are expressed in terms of Riemann's theta function, 
\begin{eqnarray} \label{theta}
\theta(\vec{z}, \tau) &=& \sum_{N \in \mathbb{Z}^g} \exp 2\pi i \left( \frac12 {^tN} \tau N + {^tN} \vec{z} \right)
\end{eqnarray}
where $\tau$ is a $g \times g$ matrix with entries $\tau_{ij} = \int_{b_j} \zeta_i$ where $\zeta_j$ are normalized holomorphic differentials and $\vec z\in\mathbb{C}^g$.
 
\begin{theorem}\label{thm:AL-exist}   There is a unique constant vector $-\kappa$, known as Riemann's constant, \cite{FarKra} VI.3.6, depending on $P_0$ such that

(i) For $n = g$,
\[W_{g, \text{smooth}}=Jac(X) \backslash (\{\vec z\in Jac(X): \theta(\vec{z}, \tau) = 0\}-\kappa).\]

(ii) For $n = g-1,$  
\[W_{g-1, \text{smooth}}=W_{g-1} \backslash (\{\vec z\in Jac(X): \theta(\vec{z}, \tau) = 0,\ \theta_{z_i}(\vec{z}, \tau) = 0 \forall i\}-\kappa).\]

(iii) for $n = 1, W_{1, \text{smooth}} = W_1$; i.e., $W_1$ is entirely smooth.
\end{theorem}
The first two statements are a direct consequence of Riemann's vanishing theorem, \cite{FarKra}, VI.3.7. 
We give a self-contained proof here of the last statement, 
which will also serve to illustrate the general result. 

\medskip

 For all points on $W_1$ to be regular means that $\Phi$ restricted to just $X$ is an embedding; i.e., that $\Phi$ is both 1:1 and an immersion on $X$. $\Phi$ on $X$ can only fail to be 1:1 if the there are two distinct points, $p_1$ and $p_2$, that map to the same point $z \in W_1$. If that were the case, by the second statement of  Theorem \ref{thm:AT}, there must exist a meromorphic function, $f$  on $X$ with a single zero at $p_1$ and a single pole at $p_2$.  
 
 Any meromorphic function, $f$, may be viewed as a non-constant holomorphic map from $X$ to the Riemann sphere $S$ in which the poles of $f$ map to the point at infinity on $S$. By the open mapping theorem this map must be onto $S$. A submersive mapping between two compact manifolds of the same dimension has a well-defined finite degree equal to the number of points, counted with multiplicity, in the inverse image, $f^{-1}(a)$,  independent of whatever point $a$ one chooses. But for $a = 0$ we already know that $f^{-1}(0) = \{p_1\}$ and so the degree of $f$ is 1. In other words $f$ is a 1:1 holomorphic map of $X$ onto $S$. But a 1:1 holomorphic map is a homeomorphism, implying that the genus of 
$X$ equals the genus of $S$ which is zero. But this contradicts our assumption that $g(X) > 0$. 

\medskip

So $\Phi$ must be  1:1. To see that it is also an immersion, observe that,
by the fundamental theorem of calculus the differential of $\Phi$ is the vector of holomorphic differentials  $\vec{\zeta}$. So this fails to be an immersion if and only if there exists a point $p \in X$ at which {\it every} holomorphic differential vanishes. Suppose there was such a point and let $\mathcal{L}_p$ be the line bundle associated to it \cite{Gun3}. Recall that, by duality,  $H^1(X, \mathcal{L}_p)$ is isomorphic to the space of holomorphic differentials vanishing at $p$ which, by our assumption, has dimension $g$. So by the Riemann-Roch theorem one has $\dim H^0(X, \mathcal{L}_p) = 1 -g +1 + g = 2$. But then again, by the argument in the previous paragraph, this would imply that $X$ has genus 0. Hence $\Phi$ is also an immersion on $X$ and therefore  all points in $W_1$ are regular. 
\qquad $\Box$

\section{Explicit representation of 
Null $\p_{a_c}''$} \label{sec:Null-expl}
We briefly recall the description of multivalued functions, $f$,  on $X$ that are multiplicative with respect to the lattice $\Lambda$ generated by the periods of $\vec{\zeta}$ as described in Section \ref{sec:GLE}. (This lattice is isomorphic to the first homology group $H_1(X, \mathbb{Z})$.) More precisely, multiplicativity means that, for elements $\gamma \in \Lambda$,
\begin{eqnarray} \label{mv}
&&f(P + \gamma) = \chi(\gamma) f(P) \,\,\,\, \chi(\gamma) \in \mathbb{C}^*,\\ \label{mult}
&&\chi(\gamma_1 + \gamma_2)  = \chi(\gamma_1) \cdot \chi(\gamma_2) 
\end{eqnarray}
 Maps, $\chi$, from $\Lambda$ to $\mathbb{C}^*$ satisfying the multiplicativity property (\ref{mult}) are called characters of $\Lambda$. If $f$ is a function for which (\ref{mv}, \ref{mult}) are both satisfied, then one says that $\chi$ is the character of the multi-valued function $f$.
\medskip

One says that a character is {\it normalized} if it takes its values in $U(1)$ and that it is {\it inessential} if it is the character of a non-vanishing, holomorphic multi-valued function. 
\medskip

 Now we fix, once and for all,  
 \begin{enumerate}
\item a point $Q_0 \in X$;
\item the associated one-point line bundle $\mathcal{L}_{Q_0}$ (see Theorem \ref{thm:AL-exist} (iii));
\item a holomorphic section, $s_0(P)$, of $\mathcal{L}_{Q_0}$, unique up to an overall constant.
\end{enumerate}
We also make use of the following constructive result:

\begin{theorem}[\cite{FarKra} III.9.10] \label{explicit} Every divisor D of degree zero is the divisor of a unique (up to a multiplicative constant) multiplicative multivalued function belonging to a unique
normalized character.
\end{theorem}

This leads to yet another description of line bundles associated to divisors of degree zero, \cite{FarKra} III.9.16:
\begin{eqnarray} \label{char}
Jac(X) & \simeq & \{\text{characters on $\Lambda$}\}/\{\text{inessential characters}\}\\
 & \simeq & \{\text{unitary characters on on $\Lambda$}\} \nonumber
\end{eqnarray} 
which is given by association of the Abel image $\Phi(D) \in Jac(X)$ to the unique normalized character $\chi$ specified in the theorem.
\medskip

For our case, the associated  multiplicative function that the theorem specifies is explicitly constructed as follows: consider $D = P_1 + \cdots + P_n - n Q_0$ where $P_1 + \cdots + P_n$ is an effective divisor corresponding to an admissible connection  under $I$. Then the multiplicative function belonging to character $\chi$, associated to $D$, is explicitly given by 
\begin{eqnarray}
f(P) &=& \exp \sum_{j=1}^n\int_{P_0}^{P} \tau_{P_j, Q_0} 
\end{eqnarray}
where $\tau_{P, Q}$ is the normalized differential which is holomorphic except for two simple pole of residues -1 and +1 at $P$ and 
$Q$, respectively. By normalized here one means that $\int_{a_j} \tau_{P, Q} = 0$, where the $a_j$ belong to the canonical homology basis chosen in \eqref{AJ-map}. 
 (Such differentials are called differentials of the third kind and are unique once $P$ and $Q$ are specified.) This is a direct consequence of the Riemann-Roch Theorem,  as we will now sketch. 
 
 \medskip
 
 The differential $\tau_{P, Q}$ is a section of the bundle $L = K\otimes \mathcal{L}_{P} \otimes \mathcal{L}_{Q}$ where $K$ is the canonical bundle (see Proposition \ref{prop:EF-K}).  This bundle has degree $2g$ and, by duality $\dim H^1(X,L ) = 0$. Hence by the Riemann-Roch Theorem we know that 
$\dim H^0(X,L ) = 2g - g +1 = g+1$. Normalization of the differential imposes $g$ independent conditions and so the normalized differentials of this type are unique up to a constant multiple. But since the residues at $P$ and $Q$  must be $-1$ and $+1$ respectively. This pins down the constant multiple uniquely. Note that, by the classical residue theorem \cite{A}, the sum of the two residues must be zero a priori. 
\medskip

The holomorphic section, unique up to an overall constant multiple, 
 of the holomorphic bundle associated to this divisor $P_1 + \cdots + P_n$  
is then explicitly represented as
\begin{eqnarray} \label{preBA}
\hat{\phi}(P) & = & f(P) s^n_0(P).
\end{eqnarray}
\medskip
where $s^n_0(P)$ is the $n$-fold product of the section $s_0(P)$, fixed above, with itself.
Let $a^c$ correspond to the divisor $P_1 + \cdots + P_n$ under the map $I$. By Theorem \ref{thm:bundle-holomorphiz}, 
there is a gauge transformation $g$ that conjugates $\p_{a^c}''$ to the $\bar{\partial}-$operator on the holomorphic bundle associated to $P_1 + \cdots + P_n$. It follows that the element of the kernel of $\p_{a^c}''$ that corresponds to $\hat{\phi}$ is given by
\begin{eqnarray} \label{BA}
\phi(P) & = & g(P) f(P) s^n_0(P).
\end{eqnarray} 
 This representation has the form of a {\it Baker-Akhiezer} section \cite{K}.

\begin{remark}

i) The transformation from  (\ref{preBA}) to  (\ref{BA}) can be viewed as corresponding to a change of bundle inner product. Indeed, the isomorphism (\ref{NS}) is mediated by gauge equivalences in the symmetric space $\mathcal{G}^c/\mathcal{G}$. This gauged space corresponds to changes of hermitian inner product on the bundle. Indeed constructions on each side of the isomorphism are made by fixing a bundle inner product. 

 ii) When $n = g$, the Baker-Akhiezer section \eqref{BA} may be re-expressed in terms of Riemann's theta function \eqref{theta}
as 
\beann
\phi(P) & = & g(P)\theta\left(\Phi(P) - \Phi(D) - \kappa\right).
\eeann

iii) It is interesting to inquire how the character of the explidit representation \eqref{BA} is related to the automorphy of the uniformized connections we considered in Section \ref{sec:cc-con-exist}. Because $\tau_{P ,Q}$ is normalized, the corresponding character, $\chi$, of  $f$ is unitary and corresponds precisely to $\s$ in \eqref{n-sig-automor} and $Q$ corresponds to the base point $z$ in $\rho_{n, \s}$. The holomorphic structure on $E_n$ and corresponding constant curvature connection under \eqref{NS} are determined by $s^n_0(P)$.
In this way all degrees of freedom are accounted for.
\end{remark}


\section{The bundle $E\otimes F$} \label{sec:bundleEF} 

 \begin{proposition}\label{prop:EF-K}
Let  $E$ and $F$ the bundles have the automorphies \eqref{n-automor} and  \eqref{tilde-automor}. The holomorphic line bundle $E\otimes F$ is isomorphic to $K^{\otimes b}$, where $K$ is the holomorphic cotangent bundle of the Riemann surface $X$ (the canonical bundle) and $b := \frac{n }{2g-2}$. Moreover,  $E\otimes F$ has degree $n$. \end{proposition} 
\begin{proof} Since the bundles $E$ and $F$ have the automorphies $\rho$ and $\tilde{\rho}^{-1}$ (see \eqref{n-automor} and  \eqref{tilde-automor}), respectively, the bundle $E\otimes F$ has the automorphy
\begin{align*}
\rho(s,z)\tilde{\rho}(s,z)^{-1} &=  \biggl[\frac{c\overline{z} +d}{cz + d}\biggr] ^{-b} |cz + d|^{2b} \\
& 
 = (cz + d)^{2b},
\end{align*}
 where again $b := \frac{2\pi n}{|X|} = \frac{n}{2g-2}$. So $E\otimes F$ 
  is a holomorphic line bundle with automorphy factor $(cz + d)^{2b}$. 
Since  the holomorphic cotangent bundle, $K\cong T^*_X$, of the Riemann surface $X$ has the automorphy factor $(cz + d)^{2}$, we conclude that $E\otimes F\cong K^{b}$. 

Since $(E\otimes F)^{\otimes(2g-2)}=K^n$,  $K$ has degree $2g-2$, and the degree in a tensor product is additive, we have that $
(2g-2)\deg(E\otimes F)=\deg K^n=n(2g-2).$ Hence the bundle $E\otimes F$ has degree $n$.\end{proof} 

Proposition \ref{prop:EF-K} implies, in particular that the line bundle $F$ has degree $0$.


\bibliographystyle{unsrt}
\bibliography{gaugefixing}

 \end{document}